\documentclass{article}
\usepackage{graphicx} 
\usepackage{amsmath}
\usepackage{amssymb}
\usepackage{amsfonts}
\usepackage{bm}
\usepackage{amsthm}
\usepackage{stmaryrd}
\usepackage[english]{babel}
\usepackage{dirtytalk}
\usepackage{dsfont}
\usepackage{soul}
\usepackage{xcolor}
\usepackage{cancel}
\usepackage{float}

\allowdisplaybreaks

\newtheorem{theorem}{Theorem}[section]
\newtheorem{corollary}[theorem]{Corollary}
\newtheorem{lemma}[theorem]{Lemma}
\newtheorem{proposition}[theorem]{Proposition}
\newtheorem{definition}[theorem]{Definition}
\newtheorem{remark}[theorem]{Remark}

\newtheorem*{acknowledgement}{Acknowledgement}
\newtheorem*{address}{Address}
\newtheorem{notation}[theorem]{Notation}

\usepackage{hyperref} 
\usepackage{cleveref}
\usepackage[square, numbers, sort&compress]{natbib} 
\usepackage[pagewise]{lineno}
\numberwithin{equation}{section}

\title{Non-divergence evolution operators modeled on Hörmander vector fields with Dini continuous coefficients}
\author{Matteo Faini}
\date{\begin{footnotesize} Dipartimento di Matematica, Politecnico di Milano, via Bonardi 9, 20133 Milano, Italy \end{footnotesize}}

\begin{document}

\maketitle

\begin{abstract}
In this paper we analyze operators $H = \partial_t-\sum_{i,j} a_{ij}(x,t) X_i X_j$, where the $X_i$'s are Hörmander vector fields generating a Carnot group and $A = [a_{ij}]$ is a symmetric and uniformly positive-definite matrix whose entries satisfy double Dini continuity, a strictly weaker condition than Hölder continuity. For these operators, we build a fundamental solution and show a two-sided Gaussian estimate for the latter, as well as upper Gaussian estimates for its derivatives up to weight $2$. As a consequence, we prove an existence result for the related Cauchy problem, under a Dini-type condition on the source.
\end{abstract}


\section{Introduction}
\label{sec 1}
In this paper we build a fundamental solution $\Gamma$ for nonvariational evolution operators of the (degenerate) form
\begin{equation}
\label{definition H}
    H = \partial_t - \sum_{i,j=1}^{m} a_{ij}(x,t) \, X_i X_j = \partial _t - a^{ij}(x,t) \, X_i X_j,
\end{equation}
where $X_1, ..., X_m$ are (smooth) Hörmander vector fields generating a Carnot group in $\mathbb{R}^n$ ($n>m$) and the matrix $[a_{ij}(x,t)]$ is symmetric and uniformly positive definite (in $\mathbb{R}^m$), with double Dini continuous entries in $\mathbb{R}^{n+1}$ (precise definitions are given below); this result is obtained through an adaptation of the classical parametrix method. We further show that $\Gamma$, together with its derivatives (up to weight $2$) along the vector fields appearing in $H$, satisfy Gaussian estimates similar to those enjoyed by the fundamental solutions of constant coefficients' heat operators; these yield also a reproduction formula for $\Gamma$ and existence of classical solutions for the Cauchy problem.

\vspace{2mm}

We follow a procedure known in literature as \emph{parametrix method}, firstly developed by E. E. Levi in \cite{Levi_ellittico_1}, \cite{Levi_ellittico_2} to construct fundamental solutions of uniformly elliptic operators with Hölder-continuous coefficients; later this procedure has been adapted by Friedman to obtain existence results in the context of uniformly parabolic operators (under the same regularity assumption on the coefficients), as discussed in \cite{Friedman_parabolic_PDE}. \\
In the last decades, several important extensions of the parabolic version of
Levi's method have been developed. In \cite{KFP_Polidoro}, Polidoro constructed a
fundamental solution for ultraparabolic operators of Kolmogorov-Fokker-Planck
type with H\"{o}lder continuous coefficients, exploiting the knowldedge of an
explicit fundamental solution for the model operator with constant
coefficients, built by Lanconelli and Polidoro in \cite{hypoelliptic_op}. In \cite{BLU_Levi_ultraparabolic}
Bonfiglioli, Lanconelli and Uguzzoni extended the applicability of the
procedure to nonvariational degenerate parabolic operators, still with
H\"{o}lder continuous coefficients, structured on H\"{o}rmander vector fields
on Carnot groups. This was the first case when Levi method was implemented
in a situation where the fundamental solution for the model operator with
constant coefficients was not explicitly known. In \cite{memoire_Bramanti_BLU} Bramanti, Brandolini,
Lanconelli, Uguzzoni built a fundamental solution for similar operators built
with general H\"{o}rmander vector fields (without any underlying structure of
translations and dilations), with local results. In \cite{Heat_kernels}, Biagi and
Bramanti generalized the results in \cite{memoire_Bramanti_BLU} again to similar operators built on
homogeneous H\"{o}rmander vector fields (without any underlying group
structure of translations), with global results. A different type of
generalization of the parabolic parametrix method is related to the weakening
of the regularity assumptions on the leading coefficients. In \cite{Regularity_Kolmogorov}, Lucertini,
Pagliarani and Pascucci built a fundamental solution for
Kolmogorov-Fokker-Planck operators of the kind studied in \cite{hypoelliptic_op}, but with
coefficients $a_{ij}\left(  x,t\right)  $ H\"{o}lder continuous in space and
only bounded measurable in time. In this case the model operator is the KFP
operator with coefficients $a_{ij}\left(  t\right)  $ only depending on time,
in a bounded measurable way; for these operators, an explicit fundamental
solution has been previously built in \cite{Fund_sol_KFP}. Recently, in \cite{Levi_Dini_parabolico}, Zhenyakova
and Cherepova adapted the parabolic parametrix method to uniformly parabolic
operators with double Dini continuous coefficients. In the present paper, we
combine the extension made in \cite{BLU_Levi_ultraparabolic} about the operator class with the
generalization in \cite{Levi_Dini_parabolico} about the coefficients' regularity, in order to prove
the existence of a fundamental solution for $H$ and the relative Gaussian
estimates under these weaker assumptions on the considered differential operator.

$\\$
In order to state in detail our assumptions, we briefly recall some notions. We call \emph{Carnot group} on $\mathbb{R}^n$ an algebraic structure $(\mathbb{R}^n, \circ, \delta_{\lambda})$ with:
\begin{itemize}
    \item [(i)] $\circ$ Lie group law on $\mathbb{R}^n$ (that is, a group law such that both the \say{translation} $(x,y) \mapsto x \circ y$ and the inversion $x \mapsto x^{-1}$ are smooth maps);
    \item [(ii)] $\{\delta_{\lambda}\}_{\lambda>0}$ family of Lie group automorphisms of $(\mathbb{R}^n, \circ)$ (\say{dilations}) acting as $\delta_{\lambda}(x)=(\lambda^{\sigma_1} x_1,..., \lambda^{\sigma_n} x_n)$, for integers $1 = \sigma_1 \leq ... \leq \sigma_n$;
    \item [(iii)] the Lie algebra $\mathfrak{a}$ of left-invariant smooth vector fields on $(\mathbb{R}^n, \circ)$ stratified as $\mathfrak{a}=\mathfrak{a}_1 \oplus \mathfrak{a}_2 \oplus ... \oplus \mathfrak{a}_s$, where $\mathfrak{a}_1$ is spanned by a set of $1$-homogeneous vector fields, $\mathfrak{a}_k=[\mathfrak{a_1}, \mathfrak{a}_{k-1}]$ for all $2 \leq k \leq s$ and $[\mathfrak{a}_s, \mathfrak{a}_1]=0$.
\end{itemize}
The vector fields spanning $\mathfrak{a}_1$ are called \emph{(Lie) generators} of $\mathbb{G}$, whereas the number $s$ in (iii) is the \emph{step} of $\mathbb{G}$ and stands for the length of commutators of the vector fields in $\textbf{X}=\{X_1, X_2, ..., X_m\}$ that are needed to fulfill Hörmander condition at any point. 

\vspace{2mm}

According to this notation, our first assumption is:
\begin{itemize}
    \item [(H1)] $X_1,...,X_m$ are smooth vector fields on $\mathbb{R}^n$, agreeing with the partial derivatives $\partial_{x_1}, ..., \partial_{x_m}$ (respectively) in the origin and generating a Carnot group $\mathbb{G}=(\mathbb{R}^n, \circ, \delta_{\lambda})$, with
    \begin{equation*}
        \delta_{\lambda}(x)=(\lambda^{\sigma_1} x_1, ..., \lambda^{\sigma_n} x_n) \quad \text{for integers} \: 1 = \sigma_1 \leq ... \leq \sigma_n.
    \end{equation*}
\end{itemize}
Let $||\cdot ||$ denote an \emph{homogeneous norm} on $\mathbb{G}$, namely a continuous function $||\cdot|| : \mathbb{R}^n \to [0,+\infty)$ with the following properties:
\begin{itemize}
    \item [(i)] $||x||=0$ if and only if $x=0$;
    \item [(ii)] $||\delta_{\lambda}(x)||=\lambda ||x||$ for every $x \in \mathbb{R}^n, \, \lambda>0$;
    \item [(iii)] there exists a constant $c \geq 1$ such that $||x|| \leq c ||x^{-1}||$ for all $x \in \mathbb{R}^n$;
    \item [(iv)] there exists a constant $\kappa \geq 1$ such that $||x \circ y|| \leq \kappa(||x||+||y||)$ for all $x,y \in \mathbb{R}^n$.
\end{itemize}
We introduce also the induced quasisymmetric quasidistance $d:\mathbb{R}^n \times \mathbb{R}^n \to [0,+\infty)$, $d(x,y):=||y^{-1} \circ x||$, which inherits from $||\cdot||$ the properties
\begin{itemize}
    \item [(i')] $d(x,y)=0$ if and only if $x=y$;
    \item [(ii')] $d(\delta_{\lambda}(x), \delta_{\lambda}(y))=\lambda d(x,y)$ for every $x, \, y \in \mathbb{R}^n, \, \lambda>0$;
    \item [(iii')] $d(x,y) \leq cd(y,x)$ for a constant $c \geq 1$ (\emph{quasi-symmetry});
    \item [(iv')] $d(x,z) \leq \kappa(d(x,y)+d(y,z))$ for every $x,y,z \in \mathbb{R}^n$, with $\kappa \geq 1$ (\emph{quasi-triangle inequality});
\end{itemize}
analogous results hold (in $\mathbb{R}^{n+1}$) for the associated \say{parabolic} quasidistance
\begin{align*}
    d_p((x,t),(y,s)) := d(x,y)+\sqrt{|t-s|}, \quad x,y \in \mathbb{R}^n, \: t,s \in \mathbb{R}.
\end{align*}
Finally, we define the \emph{homogeneous dimension} of $\mathbb{G}$ as 
\begin{equation*}
    Q := \sum_{i=1}^n \sigma_i,
\end{equation*}
recalling also its relation with volume growth in $\mathbb{G}$:
\begin{equation*}
    x'=\delta_{\lambda}(x) \implies dx'=\lambda^Q dx
\end{equation*}
(see e.g. \cite[]{Hörmander_operators_librone}); we finally set
\begin{equation}
    \label{definition of E(x,t)}
    E(x,t) := t^{-\frac{Q}{2}} \exp \left\{ -\frac{||x||^2}{t} \right\} \quad x \in \mathbb{R}^n, \, t>0,
\end{equation}
which in the Euclidean context, i.e. when $Q=n$, is simply the fundamental solution of the classical heat operator (up to scaling constants).

\vspace{1mm}

Let us now turn to the assumptions on the coefficients $a_{ij}$.\\
We call \emph{modulus of continuity} a continuous non-decreasing function $\omega : \mathbb{R^+} \to \mathbb{R}^+$ with $\omega(0)=0$; we say that the latter satisfies the \emph{Dini condition} if
\begin{equation}
\label{Dini omega}
    \widetilde{\omega}(r) := \int_0^r \frac{\omega(x)}{x} \, dx < +\infty \quad \forall r>0
\end{equation}
and the \emph{double Dini condition} if moreover
\begin{equation}
\label{double Dini omega}
    \widetilde{\widetilde{\omega}}(r)  := \int_0^r \frac{\widetilde{\omega}(x)}{x} \, dx =\int_{0}^{r} \frac{dy}{y} \int_{0}^y \frac{\omega(x)}{x} \, dx < +\infty \quad \forall r>0.
\end{equation}
Note that, given a uniformly continuous function $f$ in a set $\Omega$, the map
\begin{equation*}
    \omega_f(r):=\sup_{x,y \in \Omega: \, d(x,y) \leq r} |f(x)-f(y)|, \quad r \in [0,+\infty)
\end{equation*}
is precisely a modulus of continuity; such a function $f$ is said to be \emph{Dini continuous} (in $\Omega$) if $\omega_f$ satisfies (\ref{Dini omega}), and \emph{double Dini continuous} if $\omega_f$ satisfies also (\ref{double Dini omega}). Instead, $f$ is said to be \emph{locally Dini continuous} if it is Dini continuous on every compact subset of $\Omega$, and analogously for double Dini continuity.

Given $\Lambda>1$ and introduced the set of symmetric constant matrices
\begin{equation*}
    \mathcal{M}_{\Lambda} := \{B \in \mathbb{R}^{m \times m} : B=B^T, \, \Lambda^{-1} |v|^2 \leq \sum_{i,j=1}^m b_{ij} v_i v_j \leq \Lambda |v|^2 \hspace{2mm} \forall v \in \mathbb{R}^m\},
\end{equation*}
we assume that:
\begin{itemize}
    \item [(H2)] $A(z) = [a_{ij}(z)]_{1 \leq i,j \leq m} \in \mathcal{M}_{\Lambda}$ for every $z=(x,t) \in \mathbb{R}^{n+1}$ and
\begin{align}
\label{Dini continuity a_ij}
    & |a_{ij}(z)-a_{ij}(z')| \leq \omega(d_p(z,z')) \quad \forall z,z' \in \mathbb{R}^{n+1},
\end{align}
where $\omega$ is a modulus of continuity satisfying the double Dini condition (\ref{double Dini omega}) and, for suitable constants $c>0$ and $\delta \in (0,1)$, the inequality:
\begin{equation}
\label{Dini (1.2)}
    r_2^{-\delta} \, \omega(r_2) \leq c \, r_1^{-\delta} \, \omega(r_1) \quad \forall r_2 \geq r_1 > 0;
\end{equation}
\item [(H3)] the modulus of continuity $\omega$ is such that
\begin{equation}
\label{further condition omega}
    \int_{0}^r \frac{\omega^{1/2}(x)}{x} \, dx < +\infty \quad \forall r \geq 0.
\end{equation}
\end{itemize}
 Note that, in particular, (H2) implies that the coefficients $a_{ij}$ are \emph{double Dini continuous} over $\mathbb{R}^{n+1}$ w.r.t. the parabolic quasi-distance $d_p$.
 Moreover, it is easy to see that $A(z) \in \mathcal{M}_{\Lambda}$ implies $|a_{ij}(z)| \leq \Lambda$, hence
    \begin{equation*}
        |a_{ij}(z)-a_{ij}(z')| \leq min \{\omega(d_p(z,z')), 2 \Lambda \} = \omega'(d_p(z,z'))
    \end{equation*}
    where $\omega'(r) := min\{ \omega(r), 2\Lambda\}$ is a modulus of continuity satisfying (\ref{double Dini omega}) and (\ref{Dini (1.2)}); thus we can assume w.l.o.g. that 
    \begin{equation}
    \label{boundedness omega}
        \omega(r) \leq 2 \Lambda \quad \forall{r \geq 0}.
    \end{equation}
Let us make a comment about the relation between the \say{Dini assumptions} (H2)-(H3) and usual requirement of Hölder-continuity for $a_{ij}$:

\vspace{2mm}

\begin{remark}
\label{Dini strictly contains Hölder}
    Hölder continuity trivially implies double Dini continuity (as well as (\ref{Dini (1.2)})-(\ref{further condition omega})), but the converse is not true: e.g., for any $\alpha>2$, the function $f_{\alpha}(x,t):=\omega_{\alpha}(||x||+\sqrt{|t|})$ with
    \begin{equation*}
        \omega_{\alpha}(r) := 
        \begin{cases}
            0 \hspace{2cm} \text{if} \: r=0\\
            |log(r)|^{-\alpha} \hspace{0.6cm} \text{if} \:  0 < r < \frac{1}{2} \\
            |log(\frac{1}{2})|^{-\alpha} \hspace{0.5cm} \text{if} \: r \geq \frac{1}{2}
        \end{cases}
    \end{equation*}
    is double Dini continuous on $\mathbb{G} \times \mathbb{R}$, but not Hölder continuous therein. For this reason, this paper is a generalization of analogous results discussed in \cite{BLU_Levi_ultraparabolic}.
\end{remark}

\vspace{2mm}

For the sake of clarity, here below we explicitly point out some conventions that will be used throughout the whole paper.

\vspace{2mm}

 \begin{notation}
 \label{general notation}
- Whenever convenient, it is implicitly assumed that any repeated index (appearing once as apex and once as subscript) is summed from $1$ up to $m$, according to Einstein's convention (as done e.g. in (\ref{definition H})). \\
- We will not explicit the dependence of constants from our \say{data}, i.e. the numbers $\delta$, $\Lambda$ and the structure of $\mathbb{G}$; a constant that depends also on other parameters $f_1, ..., f_n$ will be denoted by $c(f_1, ..., f_n)$. Moreover, the symbol $c$ (or $c_0, c_1, ...$) may be used several times for possibly different constants; instead, symbols $\overline{c},\widehat{c}, \check{c}, \widetilde{c}$ (or $\overline{c}_i$, $\widehat{c}_i$, ...) are used to denote specific constants, fixed once and for all, hence each of them denote a unique constant at every occurrence.\\
- $x, \, \xi, \, y$ denote \say{spatial} variables in $\mathbb{R}^n$; $t,\, \tau,\, \eta$ denote time variables (in $\mathbb{R}$); $z=(x,t), \, \zeta=(\xi,\tau)$ and $\chi=(y,\eta)$ are compact notations for points in $\mathbb{R}^{n+1}$.\\
- $D$ denotes the diagonal of $\mathbb{R}^{n+1} \times \mathbb{R}^{n+1}$, i.e. $D=\{(z,\zeta) \in \mathbb{R}^{2n+2}: z =\zeta\}$.
 \end{notation}

 \vspace{2mm}

We conclude this introductory section by stating our main result:

\vspace{2mm}

\begin{theorem}[Fundamental solution and Gaussian estimates]
\label{existence and gaussian estimates fund sol H}
    Under assumptions (H1)-(H2)-(H3), the operator (\ref{definition H}) admits a \emph{fundamental solution} $\Gamma$, namely a function $\Gamma=\Gamma(x,t; \xi,\tau) : (\mathbb{R}^{n+1} \times \mathbb{R}^{n+1}) \setminus D \to \mathbb{R}$ such that:  \begin{itemize}
        \item [(i)] $\Gamma$ is continuous and admits continuous Lie derivatives up to order two along $X_1, ..., X_m$ w.r.t. $x$ and continuous derivative w.r.t. $t$ (for any $(x,t) \neq (\xi,\tau)$); furthermore, for any fixed $(\xi,\tau) \in \mathbb{R}^{n+1}$ there holds
        \begin{equation*}
            H \Gamma(\cdot,\cdot; \xi,\tau) = 0 \hspace{5mm} \text{in} \: \mathbb{R}^{n+1} \setminus \{ (\xi,\tau) \}.
        \end{equation*}
        \item  [(ii)] 
        $H\Gamma(\cdot; \zeta)=\delta_{\zeta}$ in $\mathcal{D}'(\mathbb{R}^{n+1})$: namely, for any $\varphi \in C^{\infty}_0(\mathbb{R}^{n+1})$, the function 
        \begin{equation*}
            u(z)=\int_{\mathbb{R}^{n+1}} \Gamma(z;\zeta) \, \varphi(\zeta) \, d\zeta
        \end{equation*}
        satisfies $Hu=\varphi$ pointwise in $\mathbb{R}^{n+1}$.
            \item [(iii)] $\Gamma(x,t; \xi,\tau)=0$ whenever $t \leq \tau$ and, given any $T>0$, for every $x \in \mathbb{R}^n, \, \xi \in \mathbb{R}^n, \, 0 < t-\tau \leq T$ the following Gaussian estimates hold true for $i, \, j = 1, \, 2, \, ..., \,m$:
        \begin{align}
            & \label{two sided gauss est Gamma} c_0^{-1} E(\xi^{-1} \circ x, c_1^{-1}(t-\tau)) \leq \Gamma(x,t; \xi,\tau) \leq c_0 \, E(\xi^{-1} \circ x, c_1 (t-\tau));\\
            &\nonumber 
            |X_j^{x} \Gamma(x,t; \xi,\tau)| \leq c_0 \, (t-\tau)^{-1/2} \, E(\xi^{-1} \circ x, c_1 (t-\tau));\\
            & \nonumber
            |X_i^{x} X_j^{x} \Gamma(x,t; \xi,\tau)|, \, |\partial_t \Gamma(x,t; \xi,\tau)| \leq c_0 \, (t-\tau)^{-1} \, E(\xi^{-1} \circ x, c_1 (t-\tau))
        \end{align}
        for constants $c_1$ depending only on $\Lambda, \delta, \mathbb{G}$, and $c_0$ depending also on $T$.
        \item [(iv)] The following reproduction formula holds for any $x, \, \xi \in \mathbb{R}^n$ and $\tau<s<t$:
        \begin{equation}
            \label{reproduction formula fund sol}
            \Gamma(x,t; \xi,\tau) = \int_{\mathbb{R}^n} \Gamma(x,t; y,s) \, \Gamma(y,s; \xi,\tau) \, dy.
        \end{equation}
        \item [(v)] Given $g=g(x)$ continuous and $f=f(x,t)$ continuous in $(x,t)$ and locally Dini continuous in $x$ uniformly w.r.t. $t$, provided $f, \, g$ satisfy suitable growth estimate (see (\ref{growth condition g}) and (\ref{growth condition f})) and $T>0$ is small enough, then the function
        \begin{equation*}
            u(z)=\iint_{\mathbb{R}^{n} \times (0,t)} \Gamma(z; \zeta) f(\zeta) \, d\zeta + \int_{\mathbb{R}^n} \Gamma(z; \xi,0) g(\xi) \, d\xi
        \end{equation*}
        is a classical solution to the Cauchy problem
        \begin{equation*}
            \begin{cases}
                Hu=f \hspace{8mm} \text{in} \: \mathbb{R}^n \times (0,T)\\
                u(\cdot,0)=g \hspace{5mm} \text{in} \: \mathbb{R}^n.
            \end{cases}
        \end{equation*}
        \end{itemize} 
\end{theorem}

\vspace{2mm}

This will be proved throughout Theorems \ref{estimates for Gamma}, \ref{thm homogeneous Cauchy pb}, \ref{nonnegativity and reproduction}, \ref{lower gaussian estimate gamma} and \ref{solution non hom Cauchy pb, case g=0}.
We point out that (H3) is needed only for Theorem \ref{solution non hom Cauchy pb, case g=0}, namely for the nonhomogeneous Cauchy problem associated to $H$; everything else does not need that extra assumption.

\vspace{2mm}

The rest of the paper is organized as follows: in Section \ref{sec 2} we give some preliminary notions about the control distance in Carnot groups, as well as some properties of Dini continuous functions. Section \ref{sec 3} is fully devoted to the development of the parametrix method: first we provide a general explanation of the procedure, then we implement it in our context, culminating with the construction of a fundamental solution for $H$ and the proof of the relative upper Gaussian estimates. Finally, in Section \ref{sec 4} we deduce existence results for the related Cauchy problem, both in the homogeneous and nonhomogeneous case, under Dini-type conditions on the regularity of the right hand side; we also prove uniform Gaussian estimates from below for $\Gamma$.

\vspace{2mm}

\begin{acknowledgement}
    This paper is an extract of the namesake Master Thesis, conduct by the author at
Politecnico di Milano under the supervision of professor Marco Bramanti. A sincere acknowledgement
goes to the latter for all the guidance and support that lead to the development of the present work. The author wishes to thank also the referees for having read carefully the paper, providing useful suggestions for its improvement.
\end{acknowledgement}

\vspace{2mm}

\section{Notation and preliminaries}
\label{sec 2}
In this section we introduce basic notions and useful lemmas about Carnot groups, as well as convenient results for moduli of continuity satisfying (H2).

\subsection{Carnot groups}
As a starting point, we recall the \say{natural} metric structure of Carnot groups. If $\mathbb{G}=(\mathbb{R}^n, \circ, \delta_{\lambda})$ is a Carnot group with Lie generators $\textbf{X}=\{X_1,...,X_m\}$, then it can be endowed with a distance $d_X$ induced by the latter; this is called \emph{Carnot-Caratheodory distance} (or \emph{control distance}) and is defined as follows:

\vspace{2mm}

\begin{definition}
    Given $X_1,...,X_m$ smooth Hörmander vector fields on $(\mathbb{R}^n,\circ)$, for any $x, \, y \in \mathbb{R}^n$ and $\delta>0$ let $C_{x,y}(\delta)$ denote the set of absolutely continuous curves $\gamma:[0,1] \to \mathbb{R}^n$ with $\gamma(0)=x, \, \gamma(1)=y$ and
    \begin{align*}
        \dot{\gamma}(t)=\alpha^i(t) \, (X_i)_{\gamma(t)} \quad \text{a.e. in} \: (0,1), \quad ||\alpha_i||_{L^{\infty}(0,1)} \leq \delta \quad \forall 1 \leq i \leq m;
    \end{align*}
    we call \emph{Carnot-Caratheodory} (or \emph{CC}, in short) \emph{distance} induced by $X_1,...,X_m$ the function $d_X:\mathbb{R}^n \times \mathbb{R}^n \to [0,+\infty), \: d_X(x,y)=\inf\{\delta>0: C_{x,y}(\delta) \neq \emptyset\}$.    
\end{definition}

\vspace{2mm}

It is well known (see \cite[Theorem 1.45]{Hörmander_operators_librone}) that $d_X$ is indeed a distance on $\mathbb{R}^n$ (in particular, any couple of points $x, \, y \in \mathbb{R}^n$ can be joined by a curve $\gamma$ as above) and it is compatible with translations and dilations of $\mathbb{G}$, namely:
\begin{align*}
    &d_X(z \circ x, z \circ y)=d_X(x,y) \quad \forall x,\, y, \, z \in \mathbb{R}^n;\\
    & d_X(\delta_{\lambda}(x), \delta_{\lambda}(y))=\lambda d_X(x,y) \quad \forall x,\, y \in \mathbb{R}^n, \, \lambda>0.
\end{align*}
The latter implies $d_X$ is induced by an homogeneous norm on $\mathbb{G}$, hence it is equivalent to the homogeneous quasidistance $d$ introduced in Section \ref{sec 1} (see e.g. \cite[Proposition 3.12]{Hörmander_operators_librone}).

\vspace{2mm}

Let us also define a regularity class of functions induced by the structure of $\mathbb{G}$; to help clarity of the definition below, we recall that the \emph{Lie derivative} along a smooth vector field $X$ of a smooth function $u$ at a point $x_0$ is defined as 
\begin{align*}
    Xu(x_0)=\lim_{h \to 0} \frac{u(x_0 \circ \mu_X(h))-u(x_0)}{h},
\end{align*}
(if the limit exists), being $\mu_X$ the solution of the Cauchy problem
\begin{equation*}
    \dot{\psi}(t)=X_{\psi(t)}, \quad \psi(0)=0.
\end{equation*}
(which is well defined, at least locally, by Cauchy theorem). From a geometrical viewpoint, $\mu_X$ is the integral curve of the vector field $X$ passing from $x_0$.

\vspace{2mm}

\begin{definition}[The class $\mathfrak{C}^2$; see \cite{BLU_Levi_ultraparabolic}, Definition 2.6]
\label{def of C2}
    Let $\Omega \subseteq \mathbb{R}^{n+1}$ be an open set. We denote by $\mathfrak{C}^2(\Omega)$ the class of functions $u(x,t)$ defined on $\Omega$ which are continuous in $\Omega$ w.r.t. the pair $(x,t)$ and such that $u(\cdot, t)$ has continuous Lie derivatives up to second order along the vector fields $X_1, ..., X_m$ (w.r.t. $x$, for every fixed $t$) and $u(x, \cdot)$ has continuous derivative (w.r.t. $t$, for every fixed $x$), in their respective domains of definition.
\end{definition}

\vspace{2mm}

We conclude this overview with the following sort of \say{Lagrange theorem} in Carnot groups (see \cite[Lemma 7.6]{BLU_uniform_gaussian_estimates}):

\vspace{2mm}

\begin{lemma}
\label{UGE lemma 7.6}
    There exists a constant $c>0$ such that
    \begin{equation*}
        |u(x_1)-u(x_2)| \leq c \, d(x_1,x_2) \, \sup_{d(x_1,y) \leq c \, d(x_1,x_2)} \, (|X_1 u(y)| + ... + |X_m u(y)|)
    \end{equation*}
    for every $x_1, \, x_2 \in \mathbb{R}^n$ and for every function $u \in C^1(\mathbb{R}^n)$.
\end{lemma}

\vspace{2mm}

\subsection{Dini continuity}

Let us now briefly deal with some useful properties of (double) Dini moduli of continuity, starting from the following basic remark:

\vspace{2mm}

\begin{remark}
    Let $\omega$ be a modulus of continuity which satisfies (\ref{double Dini omega}). Then also $r \mapsto \omega(\sqrt{r})$ and $r \mapsto \omega(cr)$ (for any $c>0$) are moduli of continuity satisfying (\ref{double Dini omega}), and there hold:
    \begin{align*}
        & \int_{0}^r \frac{\omega(c x)}{x} \, dx = \int_0^{cr} \frac{\omega(x)}{x} \, dx = \, \widetilde{\omega}(cr); \quad \int_0^r \frac{dy}{y} \int_0^y \frac{\omega(cx)}{x} \, dx = \widetilde{\widetilde{\omega}}(cr);\\
        & \int_{0}^r \frac{\omega(\sqrt{x})}{x} \, dx = \int_{0}^{\sqrt{r}} \frac{2 \omega(\rho)}{\rho} \, d\rho = 2 \widetilde{\omega}(\sqrt{r}); \quad \int_{0}^r \frac{dy}{y} \int_0^y \frac{\omega(\sqrt{x})}{x} \, dx = 4 \widetilde{\widetilde{\omega}}(\sqrt{r}).
    \end{align*}
    If, in addition, $\omega$ satisfies also (\ref{Dini (1.2)}) for some constant $\delta>0$, then 
    \begin{equation}
    \label{Dini 1.17}
        \widetilde{\omega}(r) = \int_0^r \frac{\omega(x)}{x} \, dx \geq C \, r^{-\delta} \, \omega(r) \int_0^r x^{\delta-1} \, dx = C \, \delta^{-1} \, \omega(r) \quad \forall r \geq 0, 
    \end{equation}
    namely $\omega$ can be bounded from above by $\widetilde{\omega}$ (up to a constant).
\end{remark}

\vspace{2mm}

We now turn to two less trivial results:

\vspace{2mm}

\begin{lemma}
\label{inequality omega exp}
    Let $\omega$ satisfy (\ref{double Dini omega})-(\ref{Dini (1.2)}), let $c_0>0$ be a constant. Then for any $\beta>0$ there exist constants $c_1=c_1(\beta)>0$ and $c_2>0$ such that the estimates
    \begin{align}
        & \label{Dini (1.18)}
        \omega^{\beta}(d(x,\xi)) \, E(\xi^{-1} \circ x, c_0(t-\tau)) \leq c_1 \omega^{\beta}(\sqrt{t-\tau}) \, E(\xi^{-1} \circ x, 2c_0(t-\tau));\\
        &
        \label{simil subadditivity}
        \omega(d(x,\xi)+\sqrt{t-\tau}) \, E(\xi^{-1} \circ x, c_0(t-\tau)) \leq c_2 \omega(2 \sqrt{t-\tau}) \, E(\xi^{-1} \circ x, 2c_0(t-\tau))
\end{align}
hold true for every $x, \, \xi \in \mathbb{R}^n, \, t>\tau$.
\end{lemma}

\vspace{2mm}

\begin{proof}
    For every $0 < s \leq r$ we have $\omega(r) \leq C (r/s)^{\delta} \omega(s)$ by (\ref{Dini (1.2)}), therefore
    \begin{align*}
    & \omega^{\beta}(r) \, \exp \left\{ -\frac{r^2}{c_0 s^2} \right\} \leq  C^{\beta} \left(\frac{r}{s} \right)^{\beta \, \delta} \omega^{\beta}(s) \,  \exp \left\{ -\frac{r^2}{c_0 s^2} \right\}. 
    \end{align*}
   Since the map $t \mapsto t^{\epsilon} \, e^{-kt}$ is bounded in $\mathbb{R}_+$ (for any $\epsilon>0, k>0$), we get
    \begin{align*}
    & \left(\frac{r}{s} \right)^{\beta \delta} \exp \left\{ -\frac{r^2}{2c_0 s^2} \right\} = \left(\frac{r^2}{s^2} \right)^{\beta \delta / 2} \exp \left\{ - \frac{1}{2c_0} \, \frac{r^2}{s^2} \right\} \leq c(\beta).
    \end{align*}
    so that, recalling (\ref{definition of E(x,t)}), we deduce that for every $0 < s \leq r$ there holds
    \begin{equation*}
        \omega^{\beta}(r) \, E(\xi^{-1} \circ x, c_0(t-\tau)) \leq c_1(\beta) \, \omega^{\beta}(s) \, E(\xi^{-1} \circ x, 2c_0(t-\tau))
    \end{equation*}
    By monotonicity of $\omega$ the above inequality holds also if $0 \leq r < s$; setting $r = d(x,\xi)$ and $s = \sqrt{t-\tau}$ we immediately get (\ref{Dini (1.18)}).
    Furthermore, by monotonicity of $\omega$, in the case $d(x,\xi) \leq \sqrt{t-\tau}$ we get
    \begin{align*}
        & \omega(d(x,\xi)+\sqrt{t-\tau}) \, E(\xi^{-1} \circ x, c_0(t-\tau)) \leq \omega(2 \sqrt{t-\tau}) \, E(\xi^{-1} \circ x, c_0(t-\tau)),
        \end{align*}
    i.e. (\ref{simil subadditivity}) holds, whereas if $d(x,\xi) > \sqrt{t-\tau}$ we find
    \begin{align*}
        & \omega(d(x,\xi)+\sqrt{t-\tau}) \, E(\xi^{-1} \circ x, c_0(t-\tau)) \leq \omega(2 d(x,\xi)) \, E(\xi^{-1} \circ x, c_0(t-\tau)).
    \end{align*}
    and (\ref{simil subadditivity}) is achieved by exploiting (\ref{Dini (1.18)}) with $\beta=1$.
\end{proof}

\vspace{2mm}

\begin{lemma}
\label{Dini lemma 1.1} 
    Let $\omega$ satisfy (\ref{double Dini omega}). For any $\epsilon \in (0,\frac{1}{2})$ and for any $T>0$ it holds
    \begin{equation*}
        \int_{0}^{T} \frac{\omega^{\frac{1}{2} + \epsilon}(x)}{x} \, dx < +\infty
    \end{equation*}
\end{lemma}

\vspace{2mm}

   For the proof see \cite{Baderko_potential_parabolic}, where the same proof is carried out under the further assumption that $\omega$ is subadditive (which, however, does not play a role in this result). 

   \vspace{2mm}

   \begin{remark}
       In view of the above result, the extra assumption (H3) is \say{almost harmless}, and it is actually satisfied by the simplest example of modulus of continuity which satisfies (\ref{double Dini omega}) without being of Hölder type (that is, the one in Remark \ref{Dini strictly contains Hölder}), hence this condition does not restrict our framework to the one treated in \cite{BLU_Levi_ultraparabolic}.
   \end{remark}

\vspace{2mm}

\section{Construction of the fundamental solution}
\label{sec 3}

\subsection{General idea of the parametrix method}
\label{subsec 3.1}

As mentioned in Notation \ref{general notation}, we may denote points of $\mathbb{R}^{n+1}$ by $z = (x,t)$, $\zeta = (\xi,\tau)$ and $\chi = (y,\eta)$; until not specified otherwise, we assume $\tau < \eta < t$.
For fixed $\zeta\in \mathbb{R}^{n+1}$, we consider the constant coefficients' operator
\begin{equation*}
    H_{\zeta} := \partial_t-a^{ij}(\zeta) \, X_i X_j
\end{equation*}
obtained from $H$ by \say{freezing} the coefficients $a_{ij}$ in the point $\zeta$. A deep theory about the fundamental solution of $H_{\zeta}$ (typically called \emph{parametrix} for $H$) has already been developed and is summarized in the following two theorems:

\vspace{2mm}

\begin{theorem}[The parametrix and its properties]
\label{properties freezed gamma}
    For any fixed $\zeta \in \mathbb{R}^{n+1}$, the operator $H_{\zeta}$ admits a fundamental solution $\Gamma_{\zeta}$ given by $\Gamma_{\zeta}(x,t; y,s) = \gamma_{\zeta}(y^{-1} \circ x, t - s)$, with $\gamma_{\zeta} : \mathbb{R}^{n+1} \setminus \{0\} \to \mathbb{R}$ non-negative smooth function such that:
    \begin{itemize}
        \item [(i)] $\gamma(x,t)=0$ if and only if $t \leq 0$, and $\gamma_{\zeta}(x,t) \to 0$ as $||x||+\sqrt{|t|} \to +\infty$; 
        \item [(ii)]
        \begin{equation}
        \label{BLU (1.5)}
        \int_{\mathbb{R}^n} \gamma_{\zeta}(x,t) \, dx = 1 \quad \quad \forall t > 0;
        \end{equation}
        \item [(iii)] for any $x \in \mathbb{R}^n, \, t>0, \, \tau>0$ we have the \emph{reproduction property}
        \begin{equation}
        \label{BLU (1.6)}
        \gamma_{\zeta}(x,t + \tau) = \int_{\mathbb{R}^n} 
        \gamma_{\zeta}(y^{-1} \circ x, t) \,\gamma_{\zeta}(y, \tau) \, dy;
    \end{equation} 
        \item [(iv)] the function $\Gamma_{\zeta}(\cdot; \zeta)$ is locally integrable in $\mathbb{R}^{n+1}$ and there holds
        \begin{equation*}
            H_{\zeta} \Gamma_{\zeta}(\cdot; \zeta) = \delta_{\zeta} \quad \text{in} \: D'(\mathbb{R}^{n+1});
        \end{equation*}
        \item [(v)] \label{hom Cauchy problem freezed operator}
        for any continuous function $g : \mathbb{R}^n \to \mathbb{R}$ with
            $|g(x)| \leq M e^{\nu ||x||^2}$ in $\mathbb{R}^n$
        (for some $\nu>0, \, M>0$), there exists $T>0$ such that the function
        \begin{equation*}
            u(z) := \int_{\mathbb{R}^n} \Gamma_{\zeta}(z; y,0) \, g(y) \, dy 
        \end{equation*}
        is well defined in $\mathbb{R}^n \times (0,T)$ and solves the Cauchy problem
        \begin{equation*}
        \begin{cases}
            H_{\zeta} \, u = 0 \hspace{1.15cm} \text{in} \: \mathbb{R}^n \times (0,T) \\
            u(\cdot,0) = g \hspace{1cm} \text{in} \: \mathbb{R}^n.
        \end{cases}
        \end{equation*}
    \end{itemize}
\end{theorem}

\vspace{2mm}

For the proof see \cite[Theorem 2.1]{BLU_uniform_gaussian_estimates}.\\
In the theorem below, for any nonnegative integers $h,k$ we denote by $D_{h,k}$ a differential monomial $X_{i_1} ... X_{i_h} \partial_t^k$, with indices $i_1, ..., i_h \in \{1,..., m\}$.

\vspace{2mm}

\begin{theorem}[Gaussian estimates for the parametrix]
\label{estimate freezed gamma}
    For any nonnegative integers $h,k$ and points $(x,t) \in \mathbb{R}^n \times \mathbb{R}^+$ and $\zeta, \, \zeta' \in \mathbb{R}^{n+1}$, the following hold:
    \begin{align}
        & \label{BLU (2.2)}
        \overline{c}_0^{-1} \, E(x, \overline{c}_0^{-1} t) \leq \gamma_{\zeta}(x,t) \leq \overline{c}_0 \, E(x, \overline{c}_0 t);\\
        & \label{BLU (2.3)}
        |D^{h,k} \gamma_{\zeta}(x,t)| \leq \overline{c}_{h+2k} \cdot t^{-\frac{h}{2}-k} \, E(x, \overline{c}_{h+2k} \cdot t);\\
        &  \label{BLU (2.4)} |D^{h,k} \gamma_{\zeta}(x,t) - D^{h,k} \gamma_{\zeta'}(x,t)| \leq \overline{c}_{h+2k} \, \omega(d_p(\zeta,\zeta')) \, t^{-\frac{h}{2}-k} \, E(x, \overline{c}_{h+2k} \cdot t),
    \end{align}
    where $\{\overline{c}_q\}_{q \geq 0}$ are positive constants depending only on $\Lambda, \delta$ and $\mathbb{G}$.
\end{theorem}

\vspace{2mm}

For the proof of (\ref{BLU (2.2)}) and (\ref{BLU (2.3)}) see \cite[Theorem 2.5]{BLU_uniform_gaussian_estimates}, whereas (\ref{BLU (2.4)}) follows by assumption (H2) in view of the following result by Bramanti and Fanciullo:
\begin{align*}
    & |X_{i_1} ... X_{i_p} \partial_t^q \, \gamma_{\zeta}(x,t) - X_{i_1} ... X_{i_p} \partial_t^q \, \gamma_{\zeta'}(x,t)| \leq
    \\
    & \quad \leq c(p,q) \, ||A(\zeta) - A(\zeta')||_* \, (t-\tau)^{-\frac{p}{2}-q} \, E(x,c't)
\end{align*}
with $||\cdot||_*$ denoting the usual matrix norm (see \cite[Proposition 6.2]{BMO_estimates_Bramanti}).

\vspace{2mm}

We point out that (\ref{BLU (2.2)}) actually means that the \say{heat kernel} $E(x,t)$ and the parametrix $\gamma_{\zeta}(x,t)$ are interchangeable (for any fixed $\zeta$), up to fixed constants; 
thus, denoted by $\gamma_0$ the parametrix with pole at $\zeta=0$ (i.e., the function satisfying $H \gamma_0=\delta_0$ in distributional sense), we can express (\ref{Dini (1.18)})-(\ref{simil subadditivity})-(\ref{BLU (2.3)})-(\ref{BLU (2.4)}) in terms of $\gamma_0$ instead of $E$, namely (for $x, \, \xi \in \mathbb{R}^n$ and $t>\tau$):
\begin{align}
    & \label{Dini (1.18) with gamma} \omega^{\beta}(d(x,\xi)) \, \gamma_0(\xi^{-1} \circ x, c_0(t-\tau)) \leq c_1 \omega^{\beta}(\sqrt{t-\tau}) \, \gamma_0(\xi^{-1} \circ x, 2c_0(t-\tau));\\
    & \label{simil subadditivity omega with gamma} \omega(d_p((x,t),(\xi,\tau)) \, \gamma(\xi^{-1} \circ x, c_0(t-\tau)) \leq c_2 \omega(2\sqrt{t-\tau}) \, \gamma_0(\xi^{-1} \circ x, 2c_0(t-\tau));\\
    & \label{BLU (2.3) with gamma} |D^{h,k} \gamma_{\zeta}(x,t)| \leq \check{c}_{h+2k} \, t^{-\frac{h}{2}-k} \, \gamma_0(\xi^{-1} \circ x, \check{c}_{h+2k} \cdot t);\\
    & \nonumber |D^{h,k} \gamma_{\zeta}(x,t) - D^{h,k} \gamma_{\zeta'}(x,t)| \leq\\
    & \label{BLU (2.4) with gamma} \quad \leq \check{c}_{h+2k} \, \omega(d_p(\zeta,\zeta')) \, t^{-\frac{h}{2}-k} \, \gamma_0(\xi^{-1} \circ x, \check{c}_{h+2k} \cdot t).
\end{align}
for constants $c_1=c_1(\beta)$, $c_2>0$ and $\{\check{c}_q\}_q = \overline{c}_0\{\overline{c}_{q}\}_q$ which depend only on our data (that is $\Lambda, \delta$ and the structure of $\mathbb{G}$). Moreover, since for all $c_2>c_1>0$
\begin{align*}
    & E(x,c_1t) \leq (c_1t)^{-\frac{Q}{2}} \exp\left\{ -\frac{||x||^2}{c_2 t} \right\} = \left(\frac{c_2}{c_1} \right)^{\frac{Q}{2}} E(x,c_2t),
\end{align*}
then by (\ref{BLU (1.6)})-(\ref{BLU (2.2)}) we get that for any $c_2>c_1>0$ there holds
\begin{align*}
    & \int_{\mathbb{R}^n} \gamma_0(x, c_1 t) \gamma_0(y^{-1} \circ x, c_2 \tau) \, dy \leq \overline{c}_0^2 \int_{\mathbb{R}^n} E(x, \overline{c}_0 c_1 t) E(y^{-1} \circ x, \overline{c}_0 c_2 \tau) \, dy \\
    & \quad \leq \overline{c}_0^2 \left( \frac{c_2}{c_1} \right)^{\frac{Q}{2}} \int_{\mathbb{R}^n} E(x, \overline{c}_0 c_1 t) E(y^{-1} \circ x, \overline{c}_0 c_1 \tau) \, dy \leq \overline{c}_0^{4} \left( \frac{c_2}{c_1} \right)^{\frac{Q}{2}} \times\\
    & \quad \times \int_{\mathbb{R}^n} \gamma_0(x,\overline{c}_0^2 c_1 t) \gamma_0(y^{-1} \circ x, \overline{c_0}^2 c_1 \tau) \, dy = \overline{c}_0^{4} \left( \frac{c_2}{c_1} \right)^{\frac{Q}{2}} \gamma_0(x,\overline{c}_0^2 c_1 (t+\tau)),
\end{align*}
namely for all $c_1,c_2>0$ there exists $c_3, c_4>0$ such that
\begin{align}
\label{BLU (1.6) stronger}
    \int_{\mathbb{R}^n} \gamma_0(x,c_1 t) \gamma_0(y^{-1} \circ x, c_2 \tau) \, dy \leq c_3 \gamma_0(x,c_4(t+\tau)) \quad \forall x \in \mathbb{R}^n; \: t,\, \tau>0.
\end{align}
\\
Here below we summarize some other consequences of Theorems \ref{properties freezed gamma}-\ref{estimate freezed gamma}.

\vspace{2mm}

\begin{corollary}
\label{parametrix is continuous and its 2-derivs has null integral}
    The map $(z;\zeta) \mapsto \gamma_{\zeta}(z;\zeta)$ is continuous in $\Omega=\{z \neq \zeta\}$ and for any $x \in \mathbb{R}^n, \, t>0, \, \zeta \in \mathbb{R}^{n+1}$ the following vanishing property holds:
    \begin{align*}
        & \int_{\mathbb{R}^n} (X_i X_j \gamma_{\zeta})(y \circ x, t
        ) \, dy = 0 \quad \forall 1 \leq i,j \leq m; \quad \int_{\mathbb{R}^n} (\partial_t \gamma_{\zeta})(y \circ x, t) \, dy = 0.
    \end{align*}
\end{corollary}

\vspace{2mm}

\begin{proof}
    Concerning the continuity of $(z,\zeta) \mapsto \gamma_{\zeta}(z;\zeta)$, note that
    \begin{align*}
        |\gamma_{\zeta}(z;\zeta)-\gamma_{\zeta_0}(z_0; \zeta_0)|\leq |\gamma_{\zeta}(z;\zeta)-\gamma_{\zeta_0}(z;\zeta)|+|\gamma_{\zeta_0}(z;\zeta)-\gamma_{\zeta_0}(z_0;\zeta_0)|
    \end{align*}
    for any $(z;\zeta), \, (z_0,\zeta_0) \in \Omega$. As $(z,\zeta) \to (z_0,\zeta_0)$, the two terms vanish by (\ref{BLU (2.4)}) and the continuity of $\gamma_{\zeta_0}$ (see Theorem \ref{properties freezed gamma}), respectively; thus $(z;\zeta) \mapsto \gamma_{\zeta}(z;\zeta)$ is continuous in $(z_0;\zeta_0)$, for whatever $(z_0;\zeta_0) \in \Omega$, as claimed.\\
    As for the vanishing property of the parametrix, left-invariance of the $X_i$'s and dominated convergence (based on the Gaussian estimate (\ref{BLU (2.4)})) yield
    \begin{align*}
        & \int_{\mathbb{R}^n} (X_i X_j \gamma_{\zeta})(y \circ x,t) \, dy= \int_{\mathbb{R}^n} X_i X_j \left( x \mapsto \gamma_{\zeta}(y \circ x,t) \right) \, dy \\
        & \quad =  X_i X_j \int_{\mathbb{R}^n}  \gamma_{\zeta}(y \circ x,t) \, dy=0
    \end{align*}
    in force of (\ref{BLU (1.5)}); analogously, one gets
    \begin{equation*}
        \int_{\mathbb{R}^n} (\partial_t \gamma_{\zeta})(y \circ x,t) \, dy=\partial_t \int_{\mathbb{R}^n } \gamma_{\zeta}(y \circ x,t) \, dy = 0.
    \end{equation*}
\end{proof}

\vspace{2mm}

According to the classical parametrix method, we look for a fundamental solution of $H$ of the form
\begin{equation}
    \label{def Gamma Levi method}
    \Gamma(z; \zeta) = \Gamma_{\zeta}(z; \zeta) + \iint_{\mathbb{R}^n \times (\tau,t)} \Gamma_{\chi}(z; \chi) \, \mu(\chi, \zeta) \, d\chi
\end{equation}
for an unknown kernel function $\mu=\mu(y,\eta; \xi,\tau)$ defined whenever $\eta>\tau$. To get an ansatz for the latter, let us formally apply $H$ to both sides of (\ref{def Gamma Levi method}): differentiating the integral term under the integral sign (for spatial derivatives) or via Leibniz' integral rule (for the $t$-derivative), and denoting
\begin{equation}
\label{definition Z_1}
    Z_1(z; \zeta) := (H \Gamma_{\zeta}(\cdot;\zeta))(z) \quad \forall z \neq \zeta,
\end{equation}
for every $z=(x,t), \, \zeta=(\xi,\tau) \in \mathbb{R}^{n+1}$ with $t>\tau$ we formally get
\begin{align*}
    & 0 = Z_1(z; \zeta) +  \lim_{t \to \tau^+} \int_{\mathbb{R}^n} \Gamma_{\chi}(z; y,\eta) \, \mu(y,\eta;\zeta) \, dy + \iint_{\mathbb{R}^n \times (\tau,t)} Z_1(z; \chi) \, \mu(\chi; \zeta) \, d\chi;
\end{align*}
by point (v) in Theorem \ref{properties freezed gamma}, this formally rewrites as
\begin{equation*}
    0 = Z_1(z; \zeta) + \mu(z; \zeta) + \iint_{\mathbb{R}^n \times (\tau,t)} Z_1(z; \chi) \, \mu(\chi; \zeta) \, d\chi.
\end{equation*}
Seeing $\zeta$ as a fixed parameter in (\ref{def Gamma Levi method}) and defining the integral operator $T$ as
\begin{equation*}
    \phi \mapsto (T \phi)(z) := -\iint_{\mathbb{R}^n \times (\tau,t)} Z_1(z; \chi) \, \phi(\chi) \, d\chi,
\end{equation*}
the above equation becomes
\begin{equation*}
    -Z_1(z;\zeta) = [(I - T) \, \mu(\cdot;\zeta)](z),
\end{equation*}
hence the theory of Neumann series formally gives
\begin{equation*}
    \mu(z;\zeta) = -[(I-T)^{-1} Z_1(\cdot;\zeta)](\zeta) = -\left[\sum_{j=0}^{\infty} T^{j} Z_1(\cdot;\zeta) \right](z).
\end{equation*}
Summarizing these (formal) computations, we may expect $H$ to admit a fundamental solution expressed in the form (\ref{def Gamma Levi method}), with the kernel $\mu$ assigned by
\begin{align}
    & \label{definition mu}
    \mu(z; \zeta) := -\sum_{j=1}^{\infty} Z_j(z; \zeta) \quad \forall z, \,\zeta \in \mathbb{R}^{n+1} \: \text{with} \: t>\tau
\end{align}
where, for $j \geq 2$, the sequence $\{Z_j(z;\zeta)\}_j$ is defined recursively as
\begin{align*}
    & Z_j(\cdot;\zeta) := T^{j-1} Z_1(\cdot;\zeta) = -\iint_{\mathbb{R}^n \times (\tau,t)} Z_{1}(\cdot;\chi) Z_{j-1}(\chi;\zeta) \, d\chi.
\end{align*}
The rest of this section is devoted to justify in a rigorous way this intuition.

\subsection{Properties of the kernel $\mu$}
\label{subsec 3.2}

First of all we show the convergence of the series (\ref{definition mu}) in the set $\{t>\tau\}$, namely that $\mu$ is well defined therein. To this aim, we introduce the auxiliary sequence
\begin{align*}
    & Z_1^{\lambda}(z;\zeta):=e^{-\lambda(t-\tau)} Z_1^{\lambda}(z;\zeta),\\
    & Z_j^{\lambda}(z;\zeta)=T^{j-1} Z_1^{\lambda}(z;\zeta) = -\iint_{\mathbb{R}^n \times (\tau,t)} Z_1^{\lambda}(z;\chi) Z_{j-1}^{\lambda}(\chi;\zeta) \, d\chi
\end{align*}
for $\lambda>0$ to be chosen later; this is just a \say{shrinked} version of our sequence $\{Z_j\}_j$, as it can be easily shown that $Z_j^{\lambda}(z;\zeta)=e^{-\lambda(t-\tau)} Z_j(z;\zeta)$ for all $j \geq 1$. As already done in the parabolic case \cite{Levi_Dini_parabolico}, the strategy is to find suitable upper bounds for the $Z_j^{\lambda}$'s, and then choose $\lambda$ large enough to ensure convergence of the bounding series. Due to the recursive definition of the $Z_j^{\lambda}$'s, this can be done by induction on $j \geq 1$; as for $j=1$, by Theorem \ref{properties freezed gamma} we have
\begin{align}
\label{repr formula Z_1}
    Z_1(\cdot;\zeta)=(H-H_{\zeta}) \Gamma_{\zeta}(\cdot;\zeta)=\sum_{i,j=1}^m (a_{ij}(\cdot)-a_{ij}(\zeta)) \, X_i^x X_j^x \Gamma_{\zeta}(\cdot;\zeta),
\end{align}
and thus (\ref{Dini continuity a_ij}), (\ref{simil subadditivity omega with gamma}) and (\ref{BLU (2.3) with gamma}) yield (for $x, \, \xi \in \mathbb{R}^n, \, t>\tau$)
\begin{align}
\label{estimate Z_1^lambda}
    |Z_1^{\lambda}(z;\zeta)| \leq \check{c}_2 \frac{\omega(2 \sqrt{t-\tau})}{t-\tau} \, \gamma_0(\xi^{-1} \circ x, \check{c}_2(t-\tau)) \, e^{-\lambda(t-\tau)}.
\end{align}

\vspace{2mm}

\begin{remark}
    We point out that in the corresponding Hölder case \cite{BLU_Levi_ultraparabolic} there is no need for this \say{$\lambda$-shrink}, as the upper estimates one gets directly for $\{Z_j\}$ are already good enough to ensure convergence of the series. This is not the case for double Dini coefficients, as the lack of regularity weakens the estimates for the $Z_j^{\lambda}$'s, thus making necessary a proper shrink of the sequence.
\end{remark}

\vspace{2mm}

Through the procedure sketched above, we can prove the following:

\vspace{2mm}

\begin{theorem}[Well-definiteness of $\mu$]
\label{well definiteness mu}
For any $T>0$, the series (\ref{definition mu}) totally converges over $\{ T^{-1} \leq t-\tau \leq T\}$; moreover, for every $z, \, \zeta \in \mathbb{R}^{n+1}$ with $0 < t-\tau \leq T$ the following estimate holds true:
\begin{equation}
    \label{estimate mu}
    |\mu(z; \zeta)| \leq 
    2\check{c}_{2} e^{\lambda T/2} \omega(2 \sqrt{t - \tau}) \,
    (t - \tau)^{-1}
    \gamma_0(\xi^{-1} \circ x, \check{c}_{2}(t - \tau)),
\end{equation}
being $\check{c}_{2}>0$ the constant \say{of order 2} in (\ref{BLU (2.3) with gamma})-(\ref{BLU (2.4) with gamma}).
\end{theorem}

\vspace{2mm}

\begin{proof}
    Let us denote $\phi(\lambda,\epsilon) := 
       8 \widetilde{\omega} (2 \epsilon) + 16 \Lambda \lambda^{-1} \epsilon^{-2}$ for any $\epsilon \in (0,1)$ and $\lambda>0$. We prove the following estimates by induction on $j \geq 1$:
    \begin{equation}
    \label{estimate Z_j^lambda}
    |{Z}^{(\lambda)}_j(z; \zeta)| \leq 
    \check{c}_2^j \, [\phi(\lambda, \epsilon)]^{j-1} \, \frac{\omega(2 \sqrt{t-\tau})}{t - \tau} \, \gamma_0(\xi^{-1} \circ x, \check{c}_2(t-\tau))
    \, e^{-\lambda (t-\tau)/2};
\end{equation}
the base case $j=1$ trivially follows from (\ref{estimate Z_1^lambda}). Assuming now the thesis holds true at step $j-1$ (with $j \geq 2$), for any $\epsilon \in (0,1)$ and $\lambda>0$ the above yields
\begin{align}
    & \nonumber |Z_j^{\lambda}(z;\zeta)| \leq \check{c}_2^j \, [\phi(\lambda,\epsilon)]^{j-2} e^{-\lambda(t-\tau)/2} \int_{\tau}^t \frac{\omega(2\sqrt{t-\eta}) \omega(2\sqrt{\eta-\tau})}{(t-\eta)(\eta-\tau)} \times \\
    & \nonumber \quad \times e^{-\lambda(t-\eta)/2} \left(\int_{\mathbb{R}^n} \gamma_0(y^{-1} \circ x, \check{c}_2(t-\eta) \gamma_0(\xi^{-1} \circ y, \check{c}_2(\eta-\tau)) \, dy\right) d\eta \\
    & \label{partial est Z_lambda} \quad = \check{c}_2^j \, [\phi(\lambda,\epsilon)]^{j-2} e^{-\lambda(t-\tau)/2} \gamma_0(\xi^{-1} \circ x, \check{c}_2(t-\tau)) \cdot I(\tau,t),
\end{align}
having set
\begin{equation*}
    I=I(\tau,t) := 
    \int_{\tau}^{t} 
    \frac{\omega(2 \sqrt{t-\eta}) \, \omega(2 \sqrt{\eta - \tau})}
    {(t-\eta) (\eta - \tau)} e^{-\lambda(t-\eta)/2}
     \, d\eta.
\end{equation*}
To suitably estimate $I$, we distinguish two cases:
\begin{itemize}
    \item [(1)]
if $(t - \tau)/2 \leq \epsilon ^2$, we decompose $I$ as:
\begin{equation*}
    I = 
    \left(\, \int_{\tau}^{(t + \tau)/2} + 
    \int_{(t + \tau)/2}^{t} \, \right) \,
    \{ ... \}
     \, d\eta
     = I_1 + I_2.  
\end{equation*}
Since in $I_1$ we have $t-\tau \geq t-\eta \geq (t-\tau)/2$, from (\ref{Dini omega}) we deduce that
\begin{align*}
    & I_1 \leq 2 \,\frac{\omega(2 \sqrt{t - \tau})}{t-\tau}    
    \int_{\tau}^{(t + \tau)/2}
    \frac{\omega(2 \sqrt{\eta - \tau})}
    {\eta - \tau} \, d\eta \leq 4 \, \frac{\omega(2 \sqrt{t-\tau})}{t-\tau} \,
    \widetilde{\omega}(2 \epsilon).
\end{align*}
Similarly, since in $I_2$ there holds $t-\tau \geq \eta-\tau \geq (t-\tau)/2$, we also get
\begin{equation*}
   I_2 \leq 4 \, \frac{\omega(2 \sqrt{t-\tau})}{t-\tau} \,
    \widetilde{\omega}(2 \epsilon). 
\end{equation*} 
\item [(2)]
If instead $(t - \tau)/2 > \epsilon ^2$, then we set: 
\begin{equation*}
    I = \left( \int_{\tau}^{\tau + \epsilon ^2} + 
    \int_{\tau + \epsilon ^2}^{(t+\tau)/2} + \int_{(t+\tau)/2}^{t - \epsilon ^2} + 
    \int_{t - \epsilon ^2}^{t} \right) \{...\} \, d\eta = I'_1 + I'_2 + I'_3 + I'_4.
\end{equation*}
$I'_1$ and $I'_4$ can be treated as we did for $I_1$ and $I_2$, obtaining:
\begin{equation*}
    I'_1 + I'_4 \leq 8 (t-\tau)^{-1} \omega(2 \sqrt{t-\tau}) \,
    \widetilde{\omega}(2 \epsilon).
\end{equation*}
Let us pass to $I'_3$. From (\ref{boundedness omega}) and the fact that $t - \eta > \epsilon^2$ we have
\begin{align*}
    & I'_3 \leq 2 \epsilon^{-2} \, (t-\tau)^{-1} \, \omega(2 \sqrt{t - \tau}) \int_{(t+\tau)/2}^{t - \epsilon^2} \omega(2 \sqrt{\eta - \tau}) \, e^{-\lambda(t-\eta)/2} \, d\eta 
    \\
    & \quad \leq 4 \Lambda \epsilon^{-2} \frac{\omega(2 \sqrt{t - \tau})}{t-\tau} \int_{\tau}^{t} e^{-\lambda(t-\eta)/2} \, d\eta \leq 8 \Lambda \lambda^{-1} \epsilon^{-2} \frac{\omega(2 \sqrt{t - \tau})}{t-\tau}.
\end{align*}
A similar procedure shows the validity of the same estimate also for $I'_2$.
\end{itemize}
Combining the estimates for $I_1, \, I_2$, or $I'_1+I'_4, \, I'_2, \, I'_3$ we get
\begin{align*}
    & I \leq \frac{\omega(2 \sqrt{t - \tau})}{t-\tau} \left[8 \, \widetilde{\omega}(2 \epsilon) + 16 \Lambda \lambda^{-1} \epsilon^{-2} \right]= \frac{\omega(2 \sqrt{t - \tau})}{t-\tau} \, \phi(\lambda, \epsilon)
\end{align*}
for any $\epsilon \in (0,1)$, $\lambda > 0$. Plugging this back into (\ref{partial est Z_lambda}) we get claim (\ref{estimate Z_j^lambda}).

It is then immediate that $Z_j$ is well defined for any $j \in \mathbb{N}$ and there holds
\begin{equation}
\label{est Z_j with lambda}
    |{Z}_j(z; \zeta)| \leq 
    \check{c}_2^j \, [\phi(\lambda, \epsilon)]^{j-1} \, \frac{\omega(2 \sqrt{t-\tau})}{t - \tau} \, \gamma_0(\xi^{-1} \circ x, \check{c}_2(t-\tau)) \, e^{\lambda(t-\tau)/2}
\end{equation}
for any $\epsilon \in (0,1)$ and $\lambda>0$. It follows that
\begin{align*}
    & \sum_{j=1}^{\infty} |{Z}_j(z; \zeta)| \leq \check{c}_2 \, \frac{\omega(2 \sqrt{t - \tau})}{t - \tau} \, \gamma_0(\xi^{-1} \circ x, \check{c}_2(t - \tau)) \, 
    e^{\lambda (t - \tau)/2} \sum_{j=0}^{\infty} \left[ \check{c}_2 \phi(\lambda, \epsilon) \right]^j.
\end{align*}
Recalling $\phi(\lambda,\epsilon)=8\widetilde{\omega}(2 \epsilon)+16\Lambda \lambda^{-1} \epsilon^{-2}$, we eventually fix $\epsilon = \overline{\epsilon}$ small enough to have $\widetilde{\omega}(2 \overline{\epsilon})<(32\check{c}_2)^{-1}$ and then $\lambda = \overline{\lambda}>64\Lambda (\check{c}_2 \overline{\epsilon}^2)^{-1}$, so that $\check{c}_2 \phi(\overline{\lambda},\overline{\epsilon})<1/2$; these values yield the convergence of the geometric series at the far right hand side (to a value not greater than $2$), leading us to
\begin{align*}
    & \sum_{j=1}^{\infty} |{Z}_j(z; \zeta)| \leq 2\check{c}_2 \, \frac{\omega(2 \sqrt{t - \tau})}{t - \tau} \, \gamma_0(\xi^{-1} \circ x, \check{c}_2(t - \tau)) \, 
    e^{\overline{\lambda} (t - \tau)/2}
    \\
    & \quad \leq 2\check{c}_2 \, \frac{\omega(2 \sqrt{t - \tau})}{t - \tau} \, \gamma_0(\xi^{-1} \circ x, \check{c}_2(t - \tau)) \, 
    e^{\overline{\lambda} T/2}
\end{align*}
for all $z=(x,t), \, \zeta=(\xi,\tau) \in \mathbb{R}^{n+1}$ with $0<t-\tau \leq T$. The convergence is total in any subset $\{T^{-1} \leq t-\tau \leq T\}$ in view of the boundedness of the function
\begin{equation*}
    (x,t;\xi,\tau) \mapsto \frac{\omega(2\sqrt{t-\tau})}{t-\tau} \, \gamma_0(\xi^{-1} \circ x, \check{c}_2(t-\tau))
\end{equation*}
therein.
\end{proof}

\vspace{2mm}

To remove the \say{$\lambda$-shrink} introduced above from the estimate (\ref{est Z_j with lambda}) for $Z_j$, we fix $\lambda$ to the value $\overline{\lambda}$ of the previous proof, obtaining
\begin{equation}
\label{est Z_j without lambda}
    |Z_j(z; \zeta)| \leq \check{c}_2 e^{\overline{\lambda}T/2} \, (t-\tau)^{-1} \omega(2 \sqrt{t-\tau}) \, \gamma_0(\xi^{-1} \circ x, \check{c}_2(t-\tau))
\end{equation}
for every $z, \, \zeta \in \mathbb{R}^{n+1}$ with $0 < t-\tau \leq T$.

\vspace{2mm}

\begin{corollary}
\label{eq satisfies by mu}
   For every $z, \, \zeta \in \mathbb{R}^{n+1}$ with $t>\tau$ we have
   \begin{equation*}
       {Z}_1(z; \zeta) + 
       {\mu}(z; \zeta) +
       \iint_{\mathbb{R}^n \times (\tau,t)}
       {Z}_1(z; \chi) \,
       {\mu}(\chi; \zeta) 
       \, d\chi
       = 0.
   \end{equation*}
\end{corollary}

\vspace{2mm}

\begin{proof}
    By the total convergence of the series (\ref{definition mu}) we have:
    \begin{align*}
    & -\iint_{\mathbb{R}^n \times (\tau,t)}
    {Z}_1(z; \chi) \, 
    \mu(\chi; \zeta) 
    \, d\chi = \iint_{\mathbb{R}^n \times (\tau,t)}
    {Z}_1(z; \chi) \, \sum_{j=1}^{\infty} 
    {Z}_j(\chi; \zeta) 
    \, d\chi 
    \\
    & \quad = \sum_{j=1}^{\infty} \iint_{\mathbb{R}^n \times (\tau,t)}
    {Z}_1(z; \chi) \, {Z}_j(\chi; \zeta) 
    \, d\chi = -\sum_{j=1}^{\infty} {Z}_{j+1}(z; \zeta) = {\mu}(z; \zeta) +
    {Z}_1(z; \zeta).
    \end{align*}
\end{proof}

In order to analyze the integral term in (\ref{def Gamma Levi method}), we need to combine the Gaussian estimate (\ref{estimate mu}) with suitable regularity results for $\mu$; what we need is $\mu(z; \zeta)$ continuous for $z \neq \zeta$ and Dini continuous w.r.t. $x \in \mathbb{R}^n$ (for any fixed $\xi \in \mathbb{R}^n, \, t > \tau$). This will be the content of the following two results.

\vspace{2mm}

\begin{proposition}[Continuity of $\mu$]
\label{continuity of mu}
    $\mu(x,t; \xi,\tau)$ is continuous in $\{t>\tau\}$.
\end{proposition}

\vspace{2mm}

\begin{proof}
  Thanks to the uniform convergence of the series (\ref{definition mu}) (see Proposition \ref{well definiteness mu}) it is sufficient to prove the continuity of $Z_j$ for every $j$. If $j=1$ the claim trivially holds by Theorem \ref{properties freezed gamma} and the assumption (H2). Let us then show that the continuity of $Z_j$ implies that of $Z_{j+1}$, so that the thesis follows by induction; to this aim, we prove that the latter is a uniform limit of the functions
  \begin{equation*}
    Z_{j+1}^{\sigma}(z; \zeta) := -\iint_{\mathbb{R}^n \times (\tau+\sigma, t-\sigma)} Z_1(z; \chi) Z_j(\chi; \zeta) \, d\chi, \quad z, \zeta \in \mathbb{R}^{n+1} : t>\tau+2\sigma
  \end{equation*}
  as $\sigma \to 0^+$, in every compact subset of $\{t>\tau\}$.\\
  Note that $Z_{j+1}^{\sigma}$ is continuous by dominated convergence (based on (\ref{est Z_j without lambda})), hence the validity of the above claim is sufficient to get the continuity of $Z_{j+1}$ in the set $\{t>\tau\}$. Given $K \Subset \{t>\tau\}$, denoting 
  $I_{\sigma} := (\tau, \tau + \sigma) \cup (t - \sigma, t)$ and exploiting (\ref{BLU (1.6)}) and (\ref{est Z_j without lambda}) we obtain
  \begin{align*}
      & \sup_K \left|Z_{j+1} - Z_{j+1}^{\sigma} \right| 
      \leq \sup_{(z;\zeta) \in K}  \iint_{\mathbb{R}^n \times I_{\sigma}} \left|
      Z_1(z; \chi) \, Z_j(\chi; \zeta) \right| \, d\chi 
      \\
      & \quad \leq c(K) \sup_{(z;\zeta) \in K} \int_{I_{\sigma}} \frac{\omega (2 \sqrt{t - \eta}) \, \omega(2 \sqrt{\eta - \tau})}{(t - \eta)(\eta-\tau)} \, d\eta \int_{\mathbb{R}^n} \gamma_0(y^{-1} \circ x, \check{c}_2(t - \eta)) \times
      \\
      & \quad \times \gamma_0(\xi^{-1} \circ y, \check{c}_2(\eta - \tau)) \, dy
      \\
      & \quad \leq c(K)
      \sup_{(z;\zeta) \in K} \gamma_0(\xi^{-1} \circ x, \check{c}_2(t - \tau)) \int_{I_{\sigma}} \frac{\omega(2 \sqrt{t - \eta}) \, \omega(2 \sqrt{\eta - \tau})}
      {(t - \eta)(\eta - \tau)} d\eta.
    \end{align*}
    Setting $\rho = \sqrt{\eta-\tau}$ if $\eta \in (\tau, \tau+\sigma)$ or $\rho = \sqrt{t-\eta}$ if $\eta \in (t-\sigma,t)$ we get 
    \begin{equation*}
        \int_{I_{\sigma}} \frac{\omega(2 \sqrt{t-\eta}) \, \omega(2 \sqrt{\eta - \tau})}{(t-\eta)(\eta - \tau)} \, d\eta \leq 8 \, \frac{\omega(2 \sqrt{t-\tau})}{t-\tau} \int_{0}^{\sqrt{\sigma}} \frac{\omega(2 \rho)}{\rho} \, d\rho.
    \end{equation*}
    Therefore, by (\ref{Dini omega}) and the boundedness of $\gamma_0$ in compact subsets of $\{t>\tau\}$:
    \vspace{1mm}
    \begin{align*}
      & \sup_{(z;\zeta) \in K} \left|Z_{j+1}(z;\zeta) - Z_{j+1}^{\sigma}(z;\zeta) \right| \leq c(K) \, \widetilde{\omega}(2 \sqrt{\sigma}) \times 
      \\
      & \quad \times \sup_{(z;\zeta) \in K} \{\gamma_0(\xi^{-1} \circ x, \check{c}_2(t - \tau)) \omega(2 \sqrt{t-\tau})(t-\tau)^{-1} \} \longrightarrow 0 \quad \text{as} \: \sigma \to 0^+.
    \end{align*}
\end{proof}

\vspace{2mm}

\begin{theorem}[Dini continuity of $\mu$ in $x$]
\label{Dini continuity mu}
    Given an arbitrary $T>0$, the following estimate holds for every $x, \, x', \, \xi \in \mathbb{R}^n$, $0 < t - \tau \leq T$:
    \begin{align}
        & |\mu(x,t; \zeta) - 
        \mu(x',t; \zeta)| 
        \leq c(T) \left[ \gamma_0(\xi^{-1} \circ x, \widehat{c}_2(t - \tau)) + \gamma_0(\xi^{-1} \circ x', \widehat{c}_2(t - \tau)) \right] \times
        \nonumber
        \\
        & \label{estimate delta mu}
        \quad \times (t - \tau)^{-1} \left[\omega(2d(x,x')) + 
        \omega (2\sqrt{t - \tau}) \, \widetilde{\omega} (2d(x,x')) \right]
    \end{align}
    for a constant $\widehat{c}_2$ depending only on $\Lambda, \delta, \mathbb{G}$, and $c$ depending also on $T$.
\end{theorem}

\vspace{2mm}

The proof exploits the following lemma, which is an immediate consequence of the quasi-triangle inequality of $d$:

\vspace{2mm}

\begin{lemma}
\label{E(y,ct) <= E(x,c't)}
    For every $c_1, \, c_2>0$ there exists $c_3, \, c_4>0$ such that 
    \begin{equation*}
        \sup_{d(x,y) \leq c_2 \sqrt{t - \tau}}
        \gamma_0(\xi^{-1} \circ y, c_1 (t - \tau)) \leq 
        c_3 \cdot \gamma_0(\xi^{-1} \circ x, c_4 (t - \tau))
    \end{equation*}
    for every $(x,t), (\xi,\tau) \in \mathbb{R}^{n+1}$ with $t > \tau$.
\end{lemma}

\vspace{2mm}

\begin{proof}
    In view of the \say{equivalence} (\ref{BLU (2.2)}), we can prove the statement for $E(x,t)$ in place of $\gamma_0(x,t)$. For every $y \in \mathbb{R}^n$ with $d(x,y) \leq c_2 \sqrt{t-\tau}$ there holds
    \begin{equation*}
        d(y,\xi) \geq \frac{d(x,\xi)}{\kappa} - d(x,y) \geq \frac{d(x,\xi)}{\kappa} - 
        c_2 \sqrt{t -\tau}
    \end{equation*}
    where $\kappa \geq 1$ is the constant of the quasi-triangle inequality of $d(\cdot,\cdot)$. Therefore, recalling that $(a-b)^2 \geq \frac{a^2}{2} - b^2$ for every $a,b \in \mathbb{R}$, we deduce that
    \begin{equation*}
        \sup_{d(x,y) \leq c_2 \sqrt{t - \tau}} \exp \left\{ - \frac{d(y,\xi)^2}{c_1(t - \tau)} \right\} \leq 
        \exp \left\{ - \frac{d(x,\xi)^2}{2 c_1 \kappa^2 (t - \tau)} \right\} \exp \left\{ \frac{  c_2^2}{c_1} \right\}
    \end{equation*}
    from which the thesis trivially follows, recalling the definition (\ref{definition of E(x,t)}).
\end{proof}

\vspace{3mm}

\begin{proof}[Proof of Theorem \ref{Dini continuity mu}]
    If $d(x,x')^2 \geq (t - \tau)/2$ then the thesis immediately follows from (\ref{estimate mu}), being $\omega$ non-decreasing; hence we assume $d(x,x')^2 < (t - \tau)/2$.
    For the sake of simplicity, we introduce the following notation for the $x$-variation of a function $f$:
    \begin{equation*} 
        f(x',t; \zeta) - f(x,t; \zeta) =: \Delta_{(x,x')} f(\cdot,t; \zeta).
    \end{equation*}
    From Corollary \ref{eq satisfies by mu} it is straightforward that: \\
    \begin{equation}
    \label{repr formula delta mu}
        \Delta_{(x,x')} \mu(\cdot,t; \zeta) = \Delta_{(x,x')} Z_1(\cdot,t; \zeta) + \iint_{\mathbb{R}^n \times (\tau,t)}
        \Delta_{(x,x')} Z_1(\cdot,t; \chi) \,\mu(\chi; \zeta) \, d\chi.
    \end{equation}
   In view of (\ref{repr formula Z_1}) we can decompose $\Delta_{(x,x')}Z_1(\cdot,t; \zeta)$ as follows:
    \begin{align}
        & \Delta_{(x,x')} Z_1(\cdot,t; \zeta) = 
        \Delta_{(x,x')} a^{ij}(\cdot,t) \, (X_i^x X_j^x \Gamma_{\zeta})(x',t; \zeta) + [a^{ij}(x,t) - a^{ij}(\zeta)] \times
        \nonumber
        \\
        & \label{decomposition of delta Z_1} \quad \times \Delta_{(x,x')} X_i X_j \Gamma_{\zeta}(\cdot,t; \zeta) =  K_1(x,x',t; \zeta) + K_2(x,x',t; \zeta).
    \end{align}
    As for $K_1$, by (\ref{Dini continuity a_ij}), (\ref{BLU (2.3) with gamma}) and Lemma \ref{E(y,ct) <= E(x,c't)} we immediately get
    \begin{align}
        & \label{estimate K1} |K_1(x,x',t; \zeta)| \leq 
        \check{c}_2 \, \omega(d(x,x')) \, (t - \tau)^{-1} \, \gamma_0(\xi^{-1} \circ x', \check{c}_2(t - \tau)).
    \end{align}
   We turn to $K_2$: in view of Lemma \ref{UGE lemma 7.6} we have
    \begin{align*}
        & \left|\Delta_{(x,x')} X_i X_j \Gamma_{\zeta}(\cdot,t; \zeta) \right| \leq c_0 d(x,x') \sup_{d(x,y) \leq c_0 d(x,x')} \sum_{k=1}^m \left|(X_k^x X_i^x X_j^x \Gamma_{\zeta})(y,t; \zeta) \right|
        \\
        & \quad \leq c_1 d(x,x') \, (t-\tau)^{-\frac{3}{2}} \, \sup_{d(x,y) \leq c_0 \sqrt{t-\tau}} \gamma_0(\xi^{-1} \circ y, \check{c}_3(t-\tau))
        \\
        & \quad \leq c_2 d(x,x') \, (t-\tau)^{-\frac{3}{2}} \, \gamma_0(\xi^{-1} \circ x, c_3(t-\tau)).
    \end{align*}
    for any $1 \leq i,j \leq m$, for constants $c_2, c_3>0$.  Therefore, by (\ref{Dini continuity a_ij}) and (\ref{simil subadditivity omega with gamma})
    \begin{align}
        & \label{first estimate K2} |K_2(x,x',t; \zeta)| \leq c \omega(2 \sqrt{t-\tau}) \, d(x,x') \, (t-\tau)^{-\frac{3}{2}} \,\gamma_0(\xi^{-1} \circ x, \widetilde{c}(t-\tau)).
    \end{align}
    Finally, from (\ref{Dini (1.2)}) there holds
    \begin{align*}
        & d(x,x') \, \omega(2\sqrt{t-\tau}) = d(x,x')^{1-\delta} \, \left[ d(x,x')^{\delta} \, \omega(2\sqrt{t-\tau}) \right]
        \\
        & \quad \leq (t-\tau)^{\frac{1-\delta}{2}} \, \left[ c (t-\tau)^{\frac{\delta}{2}} \, \omega(2 d(x,x')) \right] = c (t-\tau)^{\frac{1}{2}} \, \omega(2d(x,x')),
    \end{align*}
    thus for positive constant $\widetilde{c}_0$ and $\widetilde{c}_1$ we have
    \begin{equation*}
        |K_2(x,x',t; \zeta)| \leq \widetilde{c}_0 \, \omega(2d(x,x')) \, (t-\tau)^{-1} \, \gamma_0(\xi^{-1} \circ x, \widetilde{c}_1(t-\tau)).
    \end{equation*}
    The combination of the latter with (\ref{estimate K1}) gives
\begin{equation}
    \label{estimate delta Z_1}
        |\Delta_{(x,x')} Z_1(\cdot,t; \zeta)| \leq \widetilde{c}_2 \omega(2d(x,x')) \, (t - \tau)^{-1} \, \gamma_0(\xi^{-1} \circ x, \widetilde{c}_3(t - \tau)).
    \end{equation}
    Let us now focus on the integral term in (\ref{repr formula delta mu}), which we decompose as:
    \begin{align*}
        & \int_{\tau}^t \left\{\int_{\mathbb{R}^n}
        \Delta_{(x,x')} Z_1(\cdot,t; y,\eta) \,\mu(y,\eta; \zeta) dy \right\} d\eta 
        \\
        & \quad = \left( \int_{\tau}^{(t + \tau)/2} + 
        \int_{(t + \tau)/2}^{t - d(x,x')^2} + 
        \int_{t - d(x,x')^2}^{t} \right) \left\{...\right\} \, d\eta = F_1 + F_2 + F_3
    \end{align*}
    and estimate $F_1, F_2$, $F_3$ separately.
    By (\ref{Dini omega}), (\ref{BLU (1.6) stronger}), (\ref{estimate mu}) and (\ref{estimate delta Z_1})
    \begin{align*}
        & |F_1| \leq 
        c(T) \, \omega(2d(x,x')) 
        \int_{\tau}^{(\tau + t)/2}         
        \frac{\omega(2 \sqrt{\eta - \tau})}
        {(t - \eta)(\eta - \tau)} \, d\eta
        \\
        & \quad \int_{\mathbb{R}^n} \gamma_0(y^{-1} \circ x, \widetilde{c}_3(t - \eta)) \, \gamma_0(\xi^{-1} \circ y, \check{c}_2(\eta - \tau)) \, dy  \\
        & \quad \leq 
        c(T) \, \omega (2d(x,x')) \, (t - \tau)^{-1} \, 
        \gamma_0(\xi^{-1} \circ x, c_1(t - \tau)) \int_{\tau}^{(\tau + t)/2} 
        \frac{\omega(2 \sqrt{\eta - \tau})}
        {\eta - \tau} \, d\eta
        \nonumber\\
        & \quad \leq c(T) \, \omega (2d(x,x')) \frac{\widetilde{\omega}(2 \sqrt{t-\tau})}{t - \tau} \gamma_0(\xi^{-1} \circ x, c_1(t - \tau)).
    \end{align*}    
    As for $F_3$, by (\ref{est Z_j without lambda}) and Lemma \ref{E(y,ct) <= E(x,c't)} we have
    \begin{align*}
        & |\Delta_{(x,x')} Z_1(\cdot,t; \zeta)| \leq |Z_1(x,t; \zeta)| + |Z_1(x',t; \zeta)| 
        \\
        & \quad \leq c(T) \, \omega(2\sqrt{t-\tau}) (t-\tau)^{-1} \, \left[ \gamma_0(\xi^{-1} \circ x, \widetilde{c}_3(t-\tau)) + \gamma_0(\xi^{-1} \circ x', \widetilde{c}_3(t-\tau)) \right]
        \\
        & \quad \leq c(T) \, \omega(2\sqrt{t-\tau}) \, (t-\tau)^{-1} \, \gamma_0(\xi^{-1} \circ x, c_1(t-\tau)),
    \end{align*}
    hence in view of (H2), (\ref{BLU (1.6) stronger}) and (\ref{estimate mu})
    \begin{align*}
        & |F_3| \leq c(T) \int_{t - d(x,x')^2}^{t} \frac{\omega (2 \sqrt{t - \eta}) \, \omega (2 \sqrt{\eta - \tau})}
        {(t - \eta)(\eta - \tau)} \, d\eta \int_{\mathbb{R}^n} \gamma_0(\xi^{-1} \circ y, \check{c}_2(\eta - \tau)) \times
        \\ 
        & \quad \times \gamma_0 (y^{-1} \circ x, c_1(t - \eta)) \, dy
        \\
        & \quad \leq c(T) \, \omega(2 \sqrt{t-\tau}) \, (t-\tau)^{-1} \, \gamma_0(\xi^{-1} \circ x, c_2(t - \tau))
        \int_{t - d(x,x')^2}^{t} \frac{\omega (2\sqrt{t - \eta})}
        {t - \eta} \, d\eta
        \\
        & \quad
        \leq c(T) \, \omega (2 \sqrt{t - \tau}) \, \widetilde{\omega}(2d(x,x')) \, (t - \tau)^{-1} \, 
        \gamma_0(\xi^{-1} \circ x, c_2(t - \tau)).
    \end{align*} 
   Finally, to estimate $F_2$ we resort once again to the decomposition (\ref{decomposition of delta Z_1}):
    \begin{align*}
        & F_2 = \int_{(\tau + t)/2}^{t - d(x,x')^2} \int_{\mathbb{R}^n} [K_1(x,x',t; \chi) + K_2(x,x',t;\chi)]\,  \mu(\chi; \zeta) \, d\chi = F_2^1 + F_2^2.
    \end{align*}
    As for $F_2^2$, by (\ref{BLU (1.6) stronger}), (\ref{estimate mu}) and (\ref{first estimate K2}) we obtain
    \begin{align*}
        & |F_2^2| \leq c(T) \, d(x,x') \int_{(\tau + t)/2}^{t - d(x,x')^2}
        \frac{\omega (2\sqrt{t - \eta}) \, \omega (2 \sqrt{\eta - \tau})} {(t - \eta)^{3/2}(\eta - \tau)} \, d\eta \, \times
        \\
        & \quad \times \int_{\mathbb{R}^n} \gamma_0 (y^{-1} \circ x, \widetilde{c}(t - \eta)) \, \gamma_0 (\xi^{-1} \circ y, \check{c}_2(\eta - \tau)) \, dy
        \\
        & \quad \leq c(T) \, d(x,x') \,
        \frac{\omega (2 \sqrt{t - \tau})}{t - \tau} \,
        \gamma_0(\xi^{-1} \circ x, c_1(t - \tau))
        \int_{(\tau + t)/2}^{t - d(x,x')^2}
        \frac{\omega (2\sqrt{t - \eta})} {(t - \eta)^{3/2}}
        \, d\eta.
    \end{align*}
    For $\eta \in [(\tau + t)/2, t - d(x,x')^2]$ we have $t - \eta \geq d(x,x')^2$, hence by (\ref{Dini (1.2)}):
    \begin{equation*}
        (t - \eta)^{-\delta/2} \omega (2\sqrt{t - \eta}) \leq 
        c \, d(x,x')^{-\delta} \, \omega (2d(x,x')). 
    \end{equation*}
    Recalling that $(t-\tau)/2 \geq d(x,x')^2$, it follows that
    \begin{align*}
        & \int_{(\tau + t)/2}^{t-d(x,x')^2}
        \frac{\omega (2\sqrt{t - \eta})} {(t - \eta)^{3/2}}
        \, d\eta \leq c \, \frac{\omega (2 d(x,x'))} {d(x,x')^{\delta}}
        \int_{(\tau + t)/2}^{t-d(x,x')^2}
        (t - \eta)^{-(3 - \delta)/2} \, d\eta 
        \\
        & \quad \leq c \, \frac{\omega (2d(x,x'))} {d(x,x')^{\delta}} \left[\left( \frac{t - \tau}{2} \right)^{-(1 - \delta)/2} + d(x,x')^{-(1-\delta)} \right] \leq c \, \frac{\omega (2d(x,x'))} {d(x,x')}.
    \end{align*}
    Plugging this inequality in the above estimate of $F_2^2$ we find
    \begin{equation*}
        |F_2^2| \leq c(T) \, \omega(2d(x,x')) \, \omega (2 \sqrt{t - \tau}) \, 
        (t - \tau)^{-1} \,
        \gamma(\xi^{-1} \circ x, c_1(t - \tau)).
    \end{equation*}
    Let us turn to $F_2^1$. By (\ref{BLU (1.6) stronger}), (\ref{estimate mu}) and (\ref{estimate K1}) we deduce
    \begin{align*}
        & |F_2^1| \leq c(T) \, \omega (d(x,x'))
        \int_{(\tau + t)/2}^{t-d(x,x')^2} \frac{\omega (2 \sqrt{\eta - \tau})} 
        {(t - \eta)(\eta - \tau)} \, d\eta \int_{\mathbb{R}^n} \gamma_0 (y^{-1} \circ x, \check{c}(t - \eta)) \times
        \\
        & \quad \times
        \gamma_0 (\xi^{-1} \circ y, \check{c}_2(\eta - \tau)) \, dy 
        \\
        & \quad \leq c(T) \, \omega (d(x,x')) \, 
        \frac{\omega (2 \sqrt{t - \tau})} {t - \tau} \gamma_0(\xi^{-1} \circ x, c_1(t - \tau))
        \int_{(\tau + t)/2}^{t-d(x,x')^2} \frac{1} 
        {t - \eta} \, d\eta.
    \end{align*}
   Since $d(x,x')^2 \leq t-\eta$ and $\omega$ is non decreasing, then
    \begin{align*}
        & \omega (d(x,x')) \leq \omega^{\frac{1}{2} - \alpha} (d(x,x')) \, \omega^{\frac{1}{2} + \alpha} (\sqrt{t - \eta}) \quad \forall \alpha \in \left(0,\frac{1}{4}\right).
    \end{align*}
    Therefore, in view of Lemma \ref{Dini lemma 1.1}, for any $\alpha \in \left( 0, \frac{1}{4} \right)$ it holds
    \begin{align*}
        & |F_2^1| \leq c(T) \, \omega^{\frac{1}{2} - \alpha}(d(x,x')) \, \frac{\omega (2 \sqrt{t - \tau})}{t - \tau} \, \gamma_0(\xi^{-1} \circ x, c_1(t - \tau)) \int_{\tau}^{t} \frac{\omega^{\frac{1}{2}+\alpha} (\sqrt{t - \eta})} 
        {t - \eta} \, d\eta
    \\
    & \quad\leq c(T) \, \omega^{\frac{1}{2} - \alpha}(d(x,x')) \frac{\omega(2 \sqrt{t-\tau})}{t - \tau} \, 
        \gamma_0(\xi^{-1} \circ x, c_1(t - \tau)).
    \end{align*}
    Among all terms of which the $x$-variation of $\mu$ consists, $F_2^1$ is the only one not bounded from the right-hand side of (\ref{Dini continuity mu}); nonetheless, the above estimates yield
    \begin{align}
        & |\Delta_{(x,x')} \mu(\cdot,t; \xi,\tau)| \leq c(T) \, (t-\tau)^{-1} \, \gamma_0(\xi^{-1} \circ x, c_1(t-\tau)) \, \times
        \nonumber
        \\
        &
    \label{first estimate delta mu}
        \quad \times  \left[\omega^{\frac{1}{2} - \alpha}(2d(x,x')) + \omega(2\sqrt{t-\tau}) \, \widetilde{\omega}(2 d(x,x'))\right],
    \end{align}
    which is a sort of first step towards the thesis; now the strategy is to re-estimate $F_2^1$ in view of this.
    We further decompose the latter as follows:
    \begin{align*}
        & F_2^1 = \int_{(\tau + t)/2}^{t-d(x,x')^2} \int_{\mathbb{R}^n} 
        \Delta_{(x,x')} a^{ij}(\cdot,t) X_i X_j \gamma_{\chi}(y^{-1} \circ x', t - \eta) \,  \mu(x',\eta; \zeta) \, d\chi +
        \\ 
        & \quad \int_{(\tau + t)/2}^{t-d(x,x')^2}  \int_{\mathbb{R}^n} 
        \Delta_{(x,x')} a_{ij}(\cdot,t) X_i X_j \gamma_{\chi}(y^{-1} \circ x', t - \eta) \, \Delta_{(x',y)} \mu(\cdot,\eta; \zeta) \, d\chi \\
        & \quad = V_1 + V_2
    \end{align*}
    and use (\ref{first estimate delta mu}) to bound $V_2$; with the aid of (\ref{Dini continuity a_ij}), (\ref{BLU (1.6)}), (\ref{Dini (1.18) with gamma}), (\ref{BLU (2.3) with gamma}) this gives
    \begin{align*}
        & |V_2| \leq c(T) \, \omega (d(x,x')) \, \int_{(\tau + t)/2}^{t-d(x,x')^2} \frac{d\eta}{(t - \eta)(\eta - \tau)}  \int_{\mathbb{R}^n} \gamma_0(y^{-1} \circ x', c_1(t - \eta)) \, \times
        \\
        & \quad \times \gamma_0(\xi^{-1} \circ y, c_1(\eta - \tau)) \, \left[\omega^{\frac{1}{2} - \alpha} (2d(x',y)) + \omega(2\sqrt{\eta - \tau}) \, \widetilde{\omega}(2d(x',y) \right] dy
        \\
        & \quad \leq c(T) \, \omega(d(x,x')) \, (t-\tau)^{-1} \int_{(t+\tau)/2}^{t-d(x,x')^2} [\omega^{\frac{1}{2} - \alpha} (2\sqrt{t-\eta}) + \omega(2\sqrt{\eta - \tau}) \times 
        \\
        & \quad \times \widetilde{\omega}(2\sqrt{t-\eta})] \, (t-\eta)^{-1} \, d\eta \int_{\mathbb{R}^n} \gamma_0(y^{-1} \circ x', c_1(t-\eta)) \, \gamma_0(\xi^{-1} \circ y, c_1(\eta - \tau)) \, dy
        \\
        & \quad \leq c(T) \, \omega^{1-2\alpha}(d(x,x')) \, (t-\tau)^{-1} \, \gamma_0(\xi^{-1} \circ x', c_1(t-\tau)) \times 
        \\
        & \quad \times \int_{0}^{t} \left[\omega^{\frac{1}{2} + \alpha} (2\sqrt{t-\eta}) \, +\widetilde{\omega}(2 \sqrt{t-\eta}) \right] (t-\eta)^{-1} \, d\eta.
    \end{align*}
    The last integral converges by (\ref{double Dini omega}) and Lemma \ref{Dini lemma 1.1}, thus
    \begin{equation*}
        |V_2| \leq c(T) \, \omega^{1-2\alpha}(d(x,x')) \, (t-\tau)^{-1} \, \gamma_0(\xi^{-1} \circ x', c_1(t-\tau)).
    \end{equation*}
    As for $V_1$, first of all we exploit Corollary \ref{parametrix is continuous and its 2-derivs has null integral} to state that
    \begin{align*}
        & V_1 = \int_{(\tau + t)/2}^{t-d(x,x')^2} \int_{\mathbb{R}^n} 
        \Delta_{(x,x')} a^{ij}(\cdot,t) \, \mu(x',\eta; \zeta) \, \times
        \\
        & \quad \times \left[X_i X_j \gamma_{(y,\eta)}(y^{-1} \circ x', t - \eta) - X_i X_j \gamma_{(x',\eta)}(y^{-1} \circ x',t - \eta) \right] d\chi.
    \end{align*}
    Note that, for any $\eta \in \left( \frac{t+\tau}{2}, t \right)$, in view of (\ref{definition of E(x,t)})-(\ref{BLU (2.2)}) the following holds:
    \begin{align}
         & \nonumber E(\xi^{-1} \circ x, c(\eta - \tau)) \leq 2^{Q/2} \, E(\xi^{-1} \circ x, c(t-\tau)). \\
         & \label{E(x,eta-tau) <= C E(x,t-tau)} \implies \gamma_0(\xi^{-1} \circ x, c(\eta-\tau)) \leq c' \gamma_0(\xi^{-1} \circ x, c(t-\tau)),
    \end{align}
    Together with (H2), (\ref{BLU (1.5)}), (\ref{Dini (1.18) with gamma}), (\ref{BLU (2.4) with gamma}), (\ref{estimate mu}), the above gives
    \begin{align*}
        & |V_1| \leq c(T) \, \omega (d(x,x')) 
        \int_{(\tau + t)/2}^{t-d(x,x')^2}
        \frac{\omega (2 \sqrt{\eta - \tau})}{(t - \eta)(\eta - \tau)} \,        
        \gamma_0(\xi^{-1} \circ x', c_1(\eta - \tau)) \, d\eta \, \times
        \\
        & \quad \times  \int_{\mathbb{R}^n}
        \omega (d(y,x')) \,
        \gamma_0(y^{-1} \circ x', \check{c}_2(t - \eta)) \, dy
        \\
        & \quad \leq c(T) \, \omega (d(x,x')) \, \omega (2 \sqrt{t - \tau}) \, (t - \tau)^{-1} \, \gamma_0(\xi^{-1} \circ x', c_1(t - \tau)) \, \times 
        \\
        & \quad \times \int_{(\tau+t)/2}^{t-d(x,x')^2} \frac{\omega(\sqrt{t-\eta})}{t-\eta} \, d\eta
        \int_{\mathbb{R}^n} \gamma_0(y^{-1} \circ x', \check{c}_2(t - \eta)) \, dy
        \nonumber
        \\
        &
        \quad
        \leq c(T) \, \omega (d(x,x')) \, \omega (2 \sqrt{t - \tau}) \, (t - \tau)^{-1} \, \gamma_0(\xi^{-1} \circ x', c_1(t - \tau)).
    \end{align*}
    Combining the estimates for $V_1$, $V_2$ and using Lemma \ref{E(y,ct) <= E(x,c't)} we obtain
    \begin{equation*}
        |F_2^1| \leq 
        c(T) \, (t - \tau)^{-1} \, \omega^{1 - 2\alpha} (d(x,x')) \, \gamma_0(\xi^{-1} \circ x, c_1(t - \tau)), 
    \end{equation*}
    from which it follows that
    \vspace{1mm}
    \begin{align}
        & |\Delta_{(x,x')} \mu(\cdot,t; \xi,\tau)| \leq 
        c(T) \, (t - \tau)^{-1} \, \gamma_0(\xi^{-1} \circ x, c_1(t-\tau)) \, \times
        \nonumber
        \\
        &
    \label{second estimate delta mu}
        \quad \times \left[\omega^{1 - 2\alpha} (2d(x,x')) + \omega (2\sqrt{t - \tau}) \, \widetilde{\omega} (2 d(x,x')) \right].
    \end{align}
    Iterating the above argument, to conclude we must update only the estimate of $V_2$. Following the procedure above but using (\ref{second estimate delta mu}) in place of (\ref{first estimate delta mu}), we find 
    \begin{align*}
        & |V_2| \leq c(T) \, \omega(d(x,x')) \, (t-\tau)^{-1} \, \gamma_0(\xi^{-1} \circ x', c_1(t-\tau)) \times
        \\
        & \quad \times \int_{0}^{t} \frac{\omega^{1 - 2\alpha} (2\sqrt{t-\eta}) + \widetilde{\omega}(2\sqrt{t-\eta})}{t-\eta} \, d\eta.
    \end{align*}
    Since $\alpha < \frac{1}{4}$ we can directly apply Lemma \ref{Dini lemma 1.1}, which together with (H2) gives
    \begin{equation*}
        |V_2| \leq c(T) \, \omega(d(x,x')) \, (t-\tau)^{-1} \, \gamma_0(\xi^{-1} \circ x', c_1(t-\tau)),
    \end{equation*}
    thus ending the proof, according to the remark made above.
\end{proof}

\vspace{2mm}

\hspace{-2mm} To shorten the notation, sometimes it will be convenient to express (\ref{estimate delta mu}) as
\begin{align}
    & \nonumber|\Delta_{(x,x')}\mu(\cdot,t; \zeta)| \leq c(T) (t-\tau)^{-1} \widetilde{\omega}(2d(x,x')) \times \\
    & \label{short estimate delta mu} \quad \times [\gamma_0(\xi^{-1} \circ x,\widehat{c}_2(t-\tau))+\gamma_0(\xi^{-1} \circ x',\widehat{c}_2(t-\tau))] 
\end{align}
for all $x, \, \xi \in \mathbb{R}^n$ and $0<t-\tau \leq T$ (it follows from (\ref{estimate delta mu}) by (\ref{Dini 1.17})).

\vspace{2mm}

\begin{remark}
\label{double Dini condition enters}
    The above proof could not work if the coefficients $a_{ij}(x,t)$ of $H$ were only Dini continuous: indeed we have exploited multiple times Lemma \ref{Dini lemma 1.1}, which in turn requires the double Dini condition to hold for $\omega$.
\end{remark}

\vspace{2mm}

\subsection{Differentiability of the integral term}
\label{subset 3.3}

Let us exploit estimates and regularity of the kernel $\mu$ to study well-definiteness and differentiability (along $H$) of the integral term
\begin{equation*}
    J(z; \zeta) := \iint_{\mathbb{R}^n \times (\tau,t)} \Gamma_{\chi}(z; \chi) \, \mu(\chi; \zeta) \, d\chi, \quad x \in \mathbb{R}^n, \, \xi \in \mathbb{R}^n, \, t>\tau.
\end{equation*}
Note that (\ref{Dini omega}), (\ref{BLU (2.3) with gamma}), (\ref{BLU (1.6) stronger}) and (\ref{estimate mu}) give
\begin{align*}
    & \iint_{\mathbb{R}^n \times (\tau,t)} \left|\Gamma_{\chi}(z; \chi) \, \mu(\chi; \zeta) \right| \, d\chi \leq c(T) \int_{\tau}^{t} \frac{\omega(2 \sqrt{\eta - \tau})}{\eta - \tau} \, d\eta \, \times
    \\
    & \quad \times \int_{\mathbb{R}^n}
    \gamma_0(y^{-1} \circ x, \check{c}_0(t - \eta)) \,  \gamma_0(\xi^{-1} \circ y, \check{c}_2(\eta - \tau)) \, dy
    \\
    &
    \quad \leq c(T) \, \widetilde{\omega} (2 \sqrt{t - \tau}) \, 
    \gamma_0(\xi^{-1} \circ x, c_1(t - \tau)), 
    \end{align*}
namely $J$ is well defined and the following estimate holds:
\vspace{1mm}
\begin{equation}
\label{estimate J}
    |J(z; \zeta)| \leq c(T) \, \widetilde{\omega} (2 \sqrt{t - \tau}) \, 
    \gamma_0(\xi^{-1} \circ x, \widehat{c}_3(t - \tau))
\end{equation}
for any $x \in \mathbb{R}^n, \, \xi \in \mathbb{R}^n, \, 0 < t - \tau \leq T$, for a positive constant $\widehat{c}_3$ depending only on $\Lambda, \delta$ and $\mathbb{G}$. Moreover: 

\vspace{2mm}

\begin{proposition}
\label{continuity of J}
    {$J(x,t; \xi,\tau)$} is a continuous function over $\{t>\tau\}$.
\end{proposition}

\vspace{2mm}

\begin{proof}
    As for the proof of Proposition \ref{continuity of mu}, it suffices to show that $J$ is the uniform limit (in every $K \Subset\{t>\tau\}$) of the sequence of continuous functions
    \begin{equation*}
        J^{\sigma}(z; \zeta) := \iint_{\mathbb{R}^n \times (\tau+\sigma,t-\sigma)} \Gamma_{\chi}(z; \chi) \, \mu(\chi; \zeta) \, d\chi.
    \end{equation*}
    as $\sigma \to 0^+$. Given $K \Subset \{t>\tau\}$, by (\ref{BLU (2.3) with gamma}), (\ref{BLU (1.6) stronger}) and (\ref{estimate mu}) we have
    \begin{align*}
        & \sup_{(z;\zeta) \in K} \left|J^{\sigma}(z;\zeta)- J(z;\zeta)\right| \leq \sup_{(z;\zeta) \in K} \left( \int_{\tau}^{\tau+\sigma} + \int_{t-\sigma}^t \right) \int_{\mathbb{R}^n} \left| \Gamma_{\chi}(z;\chi) \, \mu(\chi,\zeta) \right| \, d\chi 
        \\
        & \quad \leq c(K) \sup_{(z;\zeta) \in K} \left( \int_{\tau}^{\tau + \sigma} + \int_{t - \sigma}^{t} \right) \frac{\omega (2 \sqrt{\eta - \tau})}{\eta - \tau} d\eta \int_{\mathbb{R}^n} \gamma_0(y^{-1} \circ x, t - \eta) \times
        \\
        & \quad \times \gamma_0(\xi^{-1} \circ y, \check{c}_2(\eta - \tau))\, dy \leq 
        \\
        & \quad \leq c(K) \sup_{(z;\zeta) \in K} \, \left[\gamma_0(\xi^{-1} \circ x, c_1(t - \tau))
        \left( \int_{\tau}^{\tau + \sigma} + \int_{t - \sigma}^{t} \right) \frac{\omega (2 \sqrt{\eta - \tau})}{\eta - \tau} d\eta \right].
    \end{align*}
    In $K \Subset \{t>\tau\}$ we have $t-\tau$ far from zero, hence $\gamma_0$ is bounded; for $\eta \in (t-\sigma,t)$ we use $\frac{t-\tau}{2} \leq \eta-\tau \leq t-\tau$, whereas for $\eta \in (\tau,\tau+\sigma)$ we use (H2) to get
    \begin{align*}
        & \sup_{(z;\zeta) \in K} \left|J^{\sigma}(z;\zeta)- J(z;\zeta)\right| \leq c(K) \sup_{(z;\zeta) \in K} \left\{ \int_{\tau}^{\tau + \sigma} \frac{\omega (2 \, \sqrt{\eta - \tau})}{\eta - \tau} \, d\eta + \right.
        \\
        & \quad \left.\int_{t-\sigma}^t \frac{\omega (2 \sqrt{\eta - \tau})}{\eta - \tau} \, d\eta \right\} \leq c(K) \sup_{(z;\zeta) \in K} \left\{ 2\widetilde{\omega}(2 \sqrt{\sigma}) +\frac{\omega(2 \sqrt{t-\tau})}{(t-\tau)/2} \cdot \sigma \right\} \longrightarrow 0
    \end{align*}
    as $\sigma \to 0^+$, which concludes the proof.
\end{proof}
    
\vspace{2mm}

Then, according to Definition (\ref{def of C2}), we have: 

\vspace{2mm}

\begin{theorem}[Derivatives of $J$]
\label{Lie derivatives J}
    For every fixed $\zeta \in \mathbb{R}^{n+1}$, the function $J(\cdot; \zeta)$ belongs to $\mathfrak{C}^2 (\{z,\zeta \in \mathbb{R}^{n+1}: t > \tau \})$ and we have:
    \begin{align}
    & \label{X_j J}
        X_j^x J(z;\zeta) = \iint_{\mathbb{R}^n \times (\tau,t)} X_j^x \Gamma_{\chi}(z; \chi) \, \mu(\chi; \zeta) \, d\chi;\\
    & \label{X_i X_j J}
        X_i^x X_j^x J(z; \zeta) = \lim_{\epsilon \to 0} \iint_{\mathbb{R}^n \times (\tau, t-\epsilon)} X_i^x X_j^x \Gamma_{\chi}(z; \chi) \, \mu(\chi; \zeta) \, d\chi;\\
    & \label{dJ/dt}
        \partial _t J(z; \zeta) = \mu(z; \zeta) + \lim_{\epsilon \to 0} \iint_{\mathbb{R}^n \times (\tau, t-\epsilon)} \partial _t \Gamma_{\chi}(z; \chi) \, \mu(\chi; \zeta) \, d\chi.
    \end{align}
    Moreover, the following estimates hold for $x \in \mathbb{R}^n, \, \xi \in \mathbb{R}^n, \, 0 < t - \tau \leq T$:
    \begin{align}
    & \label{estimate X_j J}
        |X_j^x J(z; \zeta)| \leq c(T) \, \widetilde{\omega} (2 \sqrt{t - \tau}) \, (t - \tau)^{-\frac{1}{2}} \, E(\xi^{-1} \circ x, \widehat{c}_4(t - \tau)); \\
    & \label{estimate X_i X_j J}
        |X_i^x X_j^x J(z; \zeta)| \leq c(T) \, \widetilde{\omega}(2\sqrt{t-\tau}) \, (t-\tau)^{-1} \, E(\xi^{-1} \circ x, \widehat{c}_4(t-\tau));\\
    & \label{estimate dJ/dt}
        |\partial _t J(z; \zeta)| \leq c(T) \, \widetilde{\omega}(2 \sqrt{t-\tau}) \, (t-\tau)^{-1} \, E(\xi^{-1} \circ x, \widehat{c}_4(t-\tau))
    \end{align}
    for a constants $\widehat{c}_4$ depending only on $\Lambda, \delta, \mathbb{G}$, and $c$ depending also on $T$.
\end{theorem}

\vspace{2mm}

The proof exploits the following technical lemma:

\vspace{2mm}

\begin{lemma}[cf \cite{BLU_Levi_ultraparabolic}, Lemma 2.8]
\label{BLU lemma 2.8} 
    Let $X \in \mathfrak {g}$ and let $\{ u_j \}_j$ be a sequence of continuous functions, defined on an open subset $A \subseteq \mathbb{R}^n$, with continuous Lie derivative along X. Suppose that $u_j$ converges pointwise in A to some function $u$ and that $X u_j$ converges to some function $w$ uniformly on the compact subsets of $A$. 
    Then there exists the Lie derivative of $u$ along $X$, $Xu(x)=w(x)$, for every $x \in A$.
\end{lemma}

\vspace{2mm}

\begin{proof}[Proof of Theorem \ref{Lie derivatives J}]
    As for (\ref{X_j J}), (\ref{X_i X_j J}) and (\ref{dJ/dt}), it is enough to prove them for $x \in \mathbb{R}^n, \, \xi \in \mathbb{R}^n, \, 0 < t-\tau \leq T$, for arbitrary $T>0$. For any $\epsilon > 0$ (small enough) we introduce the function
    \begin{equation*}
        J_{\epsilon}(z; \zeta) := 
        \iint_{\mathbb{R}^n \times (\tau, t-\epsilon)} \Gamma_{\chi}(z; \chi) \, \mu(\chi; \zeta) \, d\chi.
    \end{equation*}
    It is immediate that $J_{\epsilon}(\cdot,t; \xi,\tau) \longrightarrow J(\cdot,t; \xi,\tau)$ pointwise in $\mathbb{R}^n$ as $\epsilon \to 0^+$. \\
    By Theorem \ref{properties freezed gamma} and dominated convergence (based on (\ref{BLU (2.3)}), (\ref{estimate mu})), we have
    \begin{align*}
        & X_j^x J_{\epsilon} (z; \zeta) = 
        \iint_{\mathbb{R}^n \times (\tau, t-\epsilon)} X_j^x \Gamma_{\chi}(z; \chi) \, \mu(\chi; \zeta) \, d\chi;\\
        & X_i^x X_j^x J_{\epsilon} (z; \zeta) = \iint_{\mathbb{R}^n \times (\tau, t-\epsilon)} X_i^x X_j^x \Gamma_{\chi}(z; \chi) \, \mu(\chi; \zeta) \, d\chi
    \end{align*}
    for any $x \in \mathbb{R}^n, \, \xi \in \mathbb{R}^n, \, t>\tau$. Now we introduce the function
    \begin{equation*}
        w(z; \zeta) := \iint_{\mathbb{R}^n \times (\tau,t)} X_j^x \Gamma_{\chi}(z; \chi) \, \mu(\chi; \zeta) \, d\chi \quad z, \, \zeta \in \mathbb{R}^{n+1}; \,t>\tau,
    \end{equation*}
    which, by (\ref{BLU (2.3) with gamma}) and (\ref{estimate mu}), is well defined and satisfy
    \begin{align*}
        & \sup_{x \in \mathbb{R}^n} |X_j^x J_{\epsilon}(x,t; \xi,\tau) - w(x,t; \xi,\tau)| \leq c(T) \, \sup_{x \in \mathbb{R}^n}
        \int_{t - \epsilon}^{t} \frac{\omega (2 \sqrt{\eta - \tau})} {\sqrt{t - \eta} \, (\eta - \tau)} d\eta \, \times
        \\ 
        & \quad \times \int_{\mathbb{R}^n} 
        \gamma_0(y^{-1} \circ x, \check{c}_1(t - \eta)) \, \gamma_0 (\xi^{-1} \circ y, \check{c}_2(\eta - \tau)) \, dy.
    \end{align*}
    Let us assume that $\epsilon < (t-\tau)/2$, so that $t-\epsilon > (t+\tau)/2$. By (\ref{BLU (2.2)}), (\ref{BLU (1.6) stronger}), for some $c_1>0$ and $c_2=\overline{c}_0 c_1$ we have
    \begin{align*}
        & \sup_{x \in \mathbb{R}^n} \left|X_j^x J_{\epsilon}(x,t; \zeta) - w(x,t; \zeta) \right| \leq c(T) \frac{\omega (2 \sqrt{t - \tau})} {t - \tau} \, \int_{t - \epsilon}^{t} \frac{1} {\sqrt{t - \eta}} \, d\eta \times \\
        & \quad \times \sup_{x \in \mathbb{R}^n} \gamma_0(\xi^{-1} \circ x, c_1(t - \tau)) \leq c(T) \, \frac{\omega (2 \sqrt{t - \tau})} {t - \tau} \, \sqrt{\epsilon} \times 
        \\ 
        & \quad \times \sup_{x \in \mathbb{R}^n} \, E(\xi^{-1} \circ x, c_2(t - \tau)) = c(T) (t - \tau)^{-\frac{Q}{2} - 1} \, \omega (2 \sqrt{t - \tau}) \, \sqrt{\epsilon} \longrightarrow 0
    \end{align*}
    as $\epsilon \to 0^+$, thus (\ref{X_j J}) holds in view of Lemma \ref{BLU lemma 2.8}.
    Furthermore, proceeding as for the last chain of inequalities, it is immediate to get
    \begin{align*}
        & |X_j^x J(z; \zeta)| \leq c(T) E(\xi^{-1} \circ x, c_2(t-\tau)) \int_{\tau}^t \frac{\omega(2 \sqrt{\eta - \tau})}{(t-\eta)^{1/2} (\eta - \tau)} \, d\eta.
    \end{align*}
    Decomposing $(\tau,t) = (\tau, (t+\tau)/2] \cup ((t+\tau)/2, t)$ and using (\ref{Dini omega}) we obtain:
    \begin{align*}
        & \int_{\tau}^{\frac{t+\tau}{2}} \frac{\omega(2 \sqrt{\eta - \tau})}{(t-\eta)^{\frac{1}{2}} (\eta - \tau)} \, d\eta \leq \frac{\sqrt{2}} {(t-\tau)^{\frac{1}{2}}} \int_{\tau}^{\frac{t+\tau}{2}} \frac{\omega(2  \sqrt{\eta - \tau})}{\eta - \tau} d\eta \leq \sqrt{2} \, \frac{\widetilde{\omega}(2  \sqrt{t-\tau})}{(t-\tau)^{\frac{1}{2}}}; \\
        & \int_{\frac{t+\tau}{2}}^t \frac{\omega(2  \sqrt{\eta - \tau})}{(t-\eta)^{\frac{1}{2}} (\eta - \tau)} \, d\eta \leq 2 \, \frac{\omega(2  \sqrt{t-\tau})}{t-\tau} \int_{\frac{t+\tau}{2}}^t (t-\eta)^{-1/2} d\eta \leq C \, \frac{\widetilde{\omega}(2  \sqrt{t-\tau})}{(t-\tau)^{\frac{1}{2}}},
    \end{align*}
   which plugged in the inequality above give (\ref{estimate X_j J}).
    \\
    As for (\ref{X_i X_j J}), in view of Lemma \ref{BLU lemma 2.8} it suffices to prove that
    \begin{equation}
    \label{unif conv X_i X_j Jeps}
        \sup_{x \in \mathbb{R}^n} \left|\int_{t - \epsilon}^{t} d\eta \int_{\mathbb{R}^n} X_i^x X_j^x \Gamma_{(y,\eta)}(z; y,\eta) \, \mu(y,\eta; \zeta) \, dy \right| \longrightarrow 0 \quad \text{as} \, \epsilon \to 0^+.
    \end{equation}
    To this purpose, we consider the integral
    \begin{equation}
    \label{definition I hat}
        \widehat{I} = \int_{\mathbb{R}^n} X_i^x X_j^x \Gamma_{(y,\eta)}(z; y,\eta) \, \mu(y,\eta; \zeta) \, dy;
    \end{equation}
    fixed $y_0 \in \mathbb{R}^n$, we split $\widehat{I}$ as $\widehat{I} = \widehat{I}_1 + \widehat{I}_2 + \widehat{I}_3$, with:
    \begin{align*}
        & \widehat{I}_1 = \int_{\mathbb{R}^n} X_i^x X_j^x \Gamma_{(y,\eta)}(z; y,\eta) \, [\mu(y,\eta; \zeta) - 
        \mu(y_0,\eta; \zeta)] \, dy;
        \\
        & \widehat{I}_2 = \mu(y_0,\eta; \zeta) 
        \int_{\mathbb{R}^n} X_i^x X_j^x [\Gamma_{(y,\eta)}(z; y,\eta) - \Gamma_{(y_0,\eta)}(z; y,\eta)] \, dy;
        \\
        & \widehat{I}_3 = \mu(y_0,\eta; \zeta) 
        \int_{\mathbb{R}^n} X_i^x X_j^x \Gamma_{(y_0,\eta)}(z; y,\eta) \, dy.
    \end{align*}
    Corollary \ref{parametrix is continuous and its 2-derivs has null integral} yields $\widehat{I}_3 = 0$. To estimate $\widehat{I}_1$ and $\widehat{I}_2$ we fix $y_0 = x$; from (\ref{BLU (1.5)}), (\ref{Dini (1.18) with gamma}), (\ref{BLU (2.3) with gamma}), (\ref{BLU (1.6) stronger}), (\ref{estimate delta mu}) and (\ref{E(x,eta-tau) <= C E(x,t-tau)}) we get
    \begin{align*}
        & |\widehat{I}_1| \leq c(T) (t-\eta)^{-1} (\eta - \tau)^{-1} \int_{\mathbb{R}^n} [ \omega(2 d(x,y)) + \omega(2 \sqrt{\eta-\tau})\widetilde{\omega}(2 d(x,y)) ] \times 
        \\
        & \quad \times \gamma_0(y^{-1} \circ x, \check{c}_2(t-\eta)) \, \left[\gamma_0(\xi^{-1} \circ y, \widehat{c}_2(\eta - \tau)) + \gamma_0(\xi^{-1} \circ x, \widehat{c}_2(\eta - \tau)) \right] \, dy\\
        & \leq c(T) \, \frac{ \omega(2 \sqrt{t-\eta}) + \omega(2 \sqrt{\eta-\tau}) \widetilde{\omega}(2 \, \sqrt{t-\eta})}{(t-\eta)(\eta-\tau)} \left[ \int_{\mathbb{R}^n} \gamma_0(y^{-1} \circ x, c_1(t-\eta)) \times \right.
        \\
        & \times \left. \gamma_0(\xi^{-1} \circ y, c_1(\eta - \tau)) \, dy + \gamma_0(\xi^{-1} \circ x, c_1(\eta - \tau)) \int_{\mathbb{R}^n} \gamma_0(y^{-1} \circ x, c_1(t-\eta)) \, dy \right]
        \\
        &
        \leq c(T) \, \frac{\omega (2 \sqrt{t - \eta}) + \omega(2 \sqrt{\eta-\tau}) \, \widetilde{\omega} (2 \sqrt{t - \eta})} 
        {(t - \eta)(\eta - \tau)} \, \gamma_0(\xi^{-1} \circ x, c_1(t - \tau)).
        \end{align*}
    Hence by using (\ref{definition of E(x,t)})-(\ref{BLU (2.2)}) and the fact that $\eta-\tau \geq (t-\tau)/2$ we find
    \begin{align*}
        & \sup_{x \in \mathbb{R}^n} \int_{t - \epsilon}^{t} |\widehat{I}_1| \, d\eta  \leq c(T) \, (t - \tau)^{-1 - \frac{Q}{2}} \left[\int_{t - \epsilon}^{t} \frac{\omega (2 \sqrt{t - \eta}) + \widetilde{\omega} (2 \sqrt{t - \eta})}{t - \eta} \, d\eta \right]
        \\
        & \quad = c(T) \, (t - \tau)^{-1 - \frac{Q}{2}} \left[ \widetilde{\omega}(2 \sqrt{\epsilon}) + \widetilde{\widetilde{\omega}}(2 \sqrt{\epsilon}) \right] \longrightarrow 0 \quad \text{as} \: \epsilon \to 0^+
    \end{align*}
    thanks to (H2), i.e. the double Dini continuity of the coefficients $a_{ij}$.\\
    As for $\widehat{I}_2$, from (\ref{BLU (1.5)}), (\ref{Dini (1.18) with gamma}), (\ref{BLU (2.4) with gamma}), (\ref{estimate mu}) and (\ref{E(x,eta-tau) <= C E(x,t-tau)}) we get
    \begin{align*}
        & |\widehat{I}_2| \leq c(T) \, (t-\eta)^{-1} \, (\eta - \tau)^{-1} \, \omega(2 \sqrt{\eta - \tau}) \, \gamma_0(\xi^{-1} \circ x, \check{c}_2(\eta - \tau)) \, \times
        \\
        & \quad \times \int_{\mathbb{R}^n} \omega(d(x,y)) \, \gamma_0(y^{-1} \circ x, \check{c}_2(t-\eta)) \, dy \leq c(T) \frac{\omega(\sqrt{t-\eta}) \, }{t-\eta}  \times
        \\
        & \quad \times \frac{\omega(2 \sqrt{t-\eta})}{\eta-\tau} \, \gamma_0(\xi^{-1} \circ x, \check{c}_2(t - \tau)) \, \int_{\mathbb{R}^n} \gamma_0(y^{-1} \circ x, \check{c}_2(t-\eta)) \, dy 
        \\
        &
        \quad = c(T) \, (t-\eta)^{-1} \, (\eta - \tau)^{-1} \, \omega(2 \sqrt{\eta - \tau}) \, \omega(\sqrt{t-\eta}) \, \gamma_0(\xi^{-1} \circ x, \check{c}_2(t - \tau))
    \end{align*}
    \\
    and so by (H2), (\ref{definition of E(x,t)}), (\ref{BLU (2.2)}) and the fact that $t-\eta \geq (t-\tau)/2$ we deduce
    \begin{align*}
        & \sup_{x \in \mathbb{R}^n}  \int_{t-\epsilon}^{t} |\widehat{I}_2 | \, d\eta \leq c(T) \, (t-\tau)^{-\frac{Q}{2}-1} \, \omega(2 \sqrt{t-\tau}) \int_{t-\epsilon}^{t} \frac{\omega(\sqrt{t-\eta})}{t-\eta} \, d\eta
        \\
        & \quad = c(T) \, (t-\tau)^{-\frac{Q}{2}-1} \, \widetilde{\omega}(\sqrt{\epsilon}) \longrightarrow 0 \quad \text{as} \: \epsilon \to 0^+,
    \end{align*}
    namely (\ref{unif conv X_i X_j Jeps}) holds, thus yielding (\ref{X_i X_j J}). For what concerns (\ref{estimate X_i X_j J}), we have
    \begin{align*}
        |X_i^x X_j^x J(z;\zeta)| \leq \int_{\tau}^{\frac{t+\tau}{2}} |\widehat{I}| \, d\eta + \lim_{\epsilon \to 0} \int_{\frac{t+\tau}{2}}^{t-\epsilon} (|\widehat{I}_1| + |\widehat{I}_2|) \, d\eta = \widehat{J}_1 + \widehat{J}_2.
    \end{align*}
    From the definition (\ref{definition I hat}), it is immediate that $\widehat{I}$ can be bounded as follows:
    \begin{equation*}
        |\widehat{I}| \leq c(T) \, \frac{\omega(2 \sqrt{\eta - \tau})}{(t-\eta)(\eta - \tau)} \, \gamma_0(\xi^{-1} \circ x, \check{c}_2(t-\tau))
    \end{equation*}
    and consequently
    \begin{align*}
            & \widehat{J}_1 \leq c(T) \, (t-\tau)^{-1} \, \gamma_0(\xi^{-1} \circ x, \check{c}_2(t-\tau)) \int_{\tau}^{\frac{t+\tau}{2}} \frac{\omega(2 \sqrt{\eta - \tau})}{\eta - \tau} \, d\eta =
        \\ & \quad = c(T) \, \widetilde{\omega}(2 \sqrt{t-\tau}) \, (t-\tau)^{-1} \, \gamma_0(\xi^{-1} \circ x, \check{c}_2(t-\tau)).
    \end{align*}
    Consider now the above estimates for $\widehat{I}_1$ and $\widehat{I}_2$. From the monotonicity of $\omega$ it is immediate that for any $\eta \in (\tau, \frac{t+\tau}{2})$ it holds 
    \begin{equation*}
        |\widehat{I}_1| + |\widehat{I}_2| \leq c(T) \, \frac{\omega(2 \sqrt{t-\eta}) + \omega(2 \sqrt{\eta - \tau}) \, \widetilde{\omega}(2 \sqrt{t-\eta})}{(t-\eta)(\eta - \tau)} \, \gamma_0(\xi^{-1} \circ x, c_1(t-\tau))
    \end{equation*}
    and therefore, exploiting also the property (\ref{Dini 1.17}), it follows that
    \begin{align*}
        & \widehat{J}_2 \leq c(T) \left[ \int_{\frac{t+\tau}{2}}^{t} \frac{\omega(2 \sqrt{t-\eta})}{t-\eta} \, d\eta \, + \omega(2 \sqrt{t-\tau}) \int_{\frac{t+\tau}{2}}^t \frac{\widetilde{\omega}(2 \sqrt{t-\eta})}{t-\eta} \, d\eta \right] \times
        \\
        & \quad \times (t-\tau)^{-1} \, \gamma_0(\xi^{-1} \circ x, c_1(t-\tau))   
        \\
        & \quad \leq c(T) \, \frac{\widetilde{\omega}(2 \sqrt{t-\tau}) + \omega(2 \sqrt{t-\tau}) \, \widetilde{\widetilde{\omega}}(2 \sqrt{t-\tau})}{t-\tau} \, \gamma_0(\xi^{-1} \circ x, c_1(t-\tau)) 
        \\
        & \quad \leq c(T) \, (t-\tau)^{-1} \, \widetilde{\omega}(c_2 \sqrt{t-\tau}) \, \gamma_0(\xi^{-1} \circ x, c_1(t-\tau)).
    \end{align*}
    Combining the estimates for $\widehat{J}_1$ and $\widehat{J}_2$ we get (\ref{estimate X_i X_j J}) (up to using (\ref{BLU (2.2)})).

    \vspace{3mm}
    
    Finally, let us pass to the proof of (\ref{dJ/dt}), which consists of three major steps. \\
    We start by claiming that $t \mapsto J_{\epsilon}(x, t; \xi,\tau)$ has continuous derivative given by:
    \begin{equation}
    \label{dJeps/dt}
        \partial _t J_{\epsilon} (z; \zeta) = \int_{\mathbb{R}^n} \gamma_{(y, \, t - \epsilon)} (y^{-1} \circ x, \epsilon) \mu(y,t - \epsilon; \zeta) \, dy + \int_{\tau}^{t - \epsilon} \int_{\mathbb{R}^n} \partial _t \Gamma_{\chi} (z; \chi) \mu(\chi; \zeta) \, d\chi.
    \end{equation}
    Let us prove (\ref{dJeps/dt}). Given $h \in \mathbb{R}$ (with $|h|$ small enough) we have:
    \vspace{1mm}
    \begin{align}
        & \frac{1}{h} \, \left[J_{\epsilon}(x,t+h; \zeta) - J_{\epsilon}(x,t; \zeta) \right] = \int_{t-\epsilon}^{t+h-\epsilon} \int_{\mathbb{R}^n} \frac{1}{h} \, \Gamma_{\chi} (x, t+h; \chi) \, \mu(\chi; \zeta) \, d\chi
        \nonumber
        \\
        & \quad \label{decomposition in A1, A2}
        + \int_{\tau}^{t - \epsilon} \int_{\mathbb{R}^n} \frac{1}{h} \, \left[\Gamma_{\chi}(x, t+h; \chi) - \Gamma_{\chi} (x,t; \chi) \right] \, \mu(\chi; \zeta) \, d\chi = A_1(h) + A_2(h).
    \end{align}
    Using the Lagrange mean value theorem we obtain
    \begin{equation*}
        A_2(h) = \int_{\tau}^{t - \epsilon} \int_{\mathbb{R}^n} \partial _t \Gamma_{\chi}(x, t^*; \chi) \, \mu(\chi; \zeta) \, d\chi
    \end{equation*}
    for some $t^* \in (t, t+h)$ if $h>0$, for some $t^* \in (t+h,t)$ if $h<0$.
    Then, since $\partial _t \gamma_{(y,\eta)} (y^{-1} \circ x, \cdot)$ is continuous from Theorem \ref{properties freezed gamma}, by dominated convergence (based on (\ref{BLU (2.3)}) and (\ref{estimate mu})) we deduce that
    \begin{equation}
    \label{limit of A2}
        \lim_{h \to 0} A_2(h) = \int_{\tau}^{t - \epsilon} \int_{\mathbb{R}^n} \partial _t \Gamma_{\chi}(x, t; \chi) \, \mu(\chi; \zeta) \, d\chi.
    \end{equation}
    \\
    Concerning $A_1$, through the change of variable $\eta \mapsto r = \frac{\eta - (t-\epsilon)}{h}$ we get
    \begin{equation*}
        A_1(h) = \int_{0}^{1} dr \int_{\mathbb{R}^n} \Gamma_{(y, \, t - \epsilon + rh)} (y^{-1} \circ x, \epsilon + h - rh) \, \mu(y, t - \epsilon + rh; \xi,\tau) \, dy.
    \end{equation*}
    Since $\gamma_{(\xi,\tau)}(\xi^{-1} \circ x, t-\tau)$ and $\mu(x,t; \xi,\tau)$ are continuous w.r.t. $x, \, \xi \in \mathbb{R}^n, \, t>\tau$ (by Corollary \ref{parametrix is continuous and its 2-derivs has null integral} and Proposition \ref{continuity of mu}), dominated convergence gives
    \begin{equation}
    \label{limit of A1}
        \lim_{h \to 0} A_1(h) = \int_{\mathbb{R}^n} \gamma_{(y, \, t - \epsilon)} (y^{-1} \circ x, \epsilon) \, \mu(y, t - \epsilon; \xi,\tau) \, dy,
    \end{equation}
    thus (\ref{dJeps/dt}) holds true as a consequence of (\ref{decomposition in A1, A2}), (\ref{limit of A2}), (\ref{limit of A1}). As a second step we claim that, given an arbitrary $K \Subset (\tau,+\infty)$, there holds
    \begin{equation}
    \label{unif conv time 1}
        \lim_{\epsilon \to 0^+} \sup_{t \in K} \left|\mu(x,t; \xi,\tau) - \int_{\mathbb{R}^n} \gamma_{(y, \, t - \epsilon)} (y^{-1} \circ x, \epsilon) \, \mu(y, t - \epsilon; \xi,\tau) \, dy \right| = 0.
    \end{equation}
    Let us prove the claim; note that for any $0<\epsilon<\inf_K ((t-\tau)/2)$ we have
    \begin{align*}
        & \sup_{t \in K} \left|\mu(x,t; \xi,\tau) - \int_{\mathbb{R}^n} \gamma_{(y, \, t - \epsilon)} (y^{-1} \circ x, \epsilon) \, \mu(y, t - \epsilon; \xi,\tau) \, dy \right| \leq
        \\
        & \leq \sup_{t \in K} \int_{\mathbb{R}^n} \left| \left[\gamma_{(y, \, t - \epsilon)} (y^{-1} \circ x, \epsilon) - \gamma_{(x, \, t - \epsilon)} (y^{-1} \circ x, \epsilon) \right] \, \mu(y, t - \epsilon; \xi,\tau) \right| \, dy 
        \\
        & + \sup_{t \in K} \int_{\mathbb{R}^n} \gamma_{(x, \, t - \epsilon)} (y^{-1} \circ x, \epsilon) \, \left|\mu(y, t-\epsilon; \xi,\tau) - \mu(x, t-\epsilon; \xi,\tau) \right| \, dy 
        \\
        & + \sup_{t \in K} \left|\mu(x,t-\epsilon; \xi,\tau)  \int_{\mathbb{R}^n} \gamma_{(x, \, t-\epsilon)}(y^{-1} \circ x, \epsilon) \, dy \, - \mu(x,t; \xi,\tau) \right| = \overline{S}_1 + \overline{S}_2 + \overline{S}_3.
    \end{align*}
    By (\ref{BLU (1.5)}) and the uniform continuity of $\mu(x,\cdot; \xi,\tau)$ over $K$ (for any fixed $x \in \mathbb{R}^n, \, \xi \in \mathbb{R}^n)$, it is straightforward that
    \begin{align*}
        \overline{S}_3 = \sup_{t \in K} \, |\mu(x, t-\epsilon; \xi,\tau) - \mu(x,t; \xi,\tau)| \longrightarrow 0 \quad
        \text{as} \: \epsilon \to 0^+.
    \end{align*}

    \vspace{1mm}
    
    As for $\overline{S}_1$, by (\ref{Dini (1.18) with gamma}), (\ref{BLU (2.4) with gamma}) and (\ref{BLU (1.6) stronger}) we get
    \begin{align*}
        & \overline{S}_1 \leq c(K) \, \sup_{t \in K} \Bigg[\frac{\omega (2 \sqrt{t - \tau})}{t - \epsilon - \tau} \int_{\mathbb{R}^n} \omega (d(x,y)) \, \gamma_0(\xi^{-1} \circ y, \check{c}_2 (t - \epsilon - \tau)) \times
        \\
        & \quad \times \gamma_0(y^{-1} \circ x, \check{c}_0 \epsilon) \, dy \Bigg]\leq c(K) \, \omega (\sqrt{\epsilon}) \, \sup_{t \in K} \, (t-\tau)^{-1} \int_{\mathbb{R}^n} \gamma_0(y^{-1} \circ x, \check{c}_0 \epsilon) \times \\
        & \quad \times \gamma_0(\xi^{-1} \circ y, \check{c}_2(t - \epsilon - \tau)) \, dy \leq c(K) \, \omega (\sqrt{\epsilon}) \, \sup_{t \in K} \Bigg[\frac{\gamma_0(\xi^{-1} \circ x, c_1(t - \tau))}{t - \tau} \,  \Bigg].
    \end{align*}
    and similarly, with the aid of (\ref{BLU (1.5)}), (\ref{Dini (1.18) with gamma}), (\ref{BLU (2.3) with gamma}), (\ref{BLU (1.6) stronger}) and (\ref{short estimate delta mu})
    \begin{align*}
        & \overline{S}_2 \leq c(K) \, \sup_{t \in K} \, \Bigg[(t - \epsilon - \tau)^{-1} \int_{\mathbb{R}^n} \gamma_0(y^{-1} \circ x, \epsilon) \, \widetilde{\omega} (2d(x,y)) \times
        \\
        & \quad \times [\gamma_0(\xi^{-1} \circ y, \check{c}_1 (t - \epsilon - \tau))+ \gamma_0(\xi^{-1} \circ x, \widehat{c}_2 (t - \epsilon - \tau))] \,  \, dy \Bigg]
        \\
        & \quad \leq c(K) \, \sup_{t \in K} \, \widetilde{\omega} (2 \sqrt{\epsilon}) \Bigg\{ \, (t-\tau)^{-1} \Bigg[ \int_{\mathbb{R}^n} \gamma_0(\xi^{-1} \circ y, \widehat{c}_2(t-\epsilon-\tau)) \times
        \\
        & \quad \times \gamma_0(y^{-1} \circ x, \epsilon) \, dy + \gamma_0(\xi^{-1} \circ x, \widehat{c}_2(t-\epsilon-\tau)) \int_{\mathbb{R}^n} \gamma_0(y^{-1} \circ x, \epsilon) \, dy \Bigg] \Bigg\}
        \\
        & \quad \leq c(K) \, \widetilde{\omega} (2 \sqrt{\epsilon}) \sup_{t \in K} \left[\frac{\gamma_0(\xi^{-1} \circ x, c_1(t-\tau)) + \gamma_0(\xi^{-1} \circ x, \widehat{c}_2(t-\epsilon-\tau))}{t-\tau} \right].
    \end{align*}
    Then from (\ref{E(x,eta-tau) <= C E(x,t-tau)}) and the usual compactness argument we conclude that
    \begin{align*}
        & \overline{S}_1+\overline{S}_2 \leq c(K) [ \omega(\sqrt{\epsilon})+\widetilde{\omega} (2 \sqrt{\epsilon})]\longrightarrow 0 \quad \text{as} \: \epsilon \to 0^+,
    \end{align*}
    thus (\ref{unif conv time 1}) holds true. The final step consists in proving the following:
    \begin{equation}
    \label{unif conv time 2}
        \lim_{\epsilon \to 0^+} \sup_{t \in K} \, \left|\int_{t - \epsilon}^{t} d\eta \int_{\mathbb{R}^n} \partial _t \gamma_{(y,\eta)}(y^{-1} \circ x, t - \eta) \, \mu(y,\eta; \xi,\tau) \right| \, dy= 0 
    \end{equation}
    since the proof follows exactly the lines of that of (\ref{unif conv X_i X_j Jeps}), we omit the details.
    \\
    Combining (\ref{dJeps/dt}), (\ref{unif conv time 1}), (\ref{unif conv time 2}) we finally get (\ref{dJ/dt}). As for (\ref{estimate dJ/dt}), thanks to the fact that $\partial_t \Gamma_{\zeta}(z;\zeta)$ and $X_i^x X_j^x \Gamma_{\zeta}(z;\zeta)$ enjoy the same estimate up to constants (see Theorem \ref{estimate freezed gamma}), reasoning as for the proof of (\ref{estimate X_i X_j J}) one easily gets
    \begin{align*}
        & \left|\lim_{\epsilon \to 0} \int_{\tau}^{t - \epsilon} d\eta \int_{\mathbb{R}^n} \partial _t \gamma_{(y,\eta)}(y^{-1} \circ x, t - \eta) \, \mu(y,\eta; \xi,\tau) \, dy \right| \leq
        \\
        & \quad \leq c(T) \, \widetilde{\omega}(2 \sqrt{t-\tau}) \, (t-\tau)^{-1} \, \gamma_0(\xi^{-1} \circ x, c_1(t-\tau));
    \end{align*}
    furthermore, combining the estimate (\ref{estimate mu}) with (\ref{Dini 1.17}) one easily deduces that also $\mu(z;\zeta)$ enjoys the same estimate, thus yielding (\ref{estimate dJ/dt}).
\end{proof}

\vspace{2mm}

Note that, according to definitions (\ref{def Gamma Levi method})-(\ref{definition mu}), in principle $\Gamma(x,t; \xi,\tau)$ is defined only for $t>\tau$. In what follows, we agree to extend $\Gamma$ to the whole $\mathbb{R}^{n+1}$ by setting $\Gamma(x,t; \xi,\tau) = 0 \:\,$ for $t \leq \tau$. Then, by Theorem 
\ref{Lie derivatives J} one gets:

\vspace{2mm}

\begin{theorem}[Upper Gaussian estimates for $\Gamma$]
\label{estimates for Gamma}
    $\Gamma$ is a continuous function outside the diagonal of $\mathbb{R}^{n+1} \times \mathbb{R}^{n+1}$ and for any $\zeta \in \mathbb{R}^{n+1}$ we have that
    \begin{align}
    & \label{Gamma in C2}
        \Gamma(\cdot; \zeta) \in \mathfrak{C}^2 (\mathbb{R}^{n+1} \setminus \{\zeta\});\\
    & \label{H Gamma = 0}
        H (\, \Gamma(\cdot; \zeta) \,) = 0 \quad \text{in} \:\, \mathbb{R}^{n+1} \setminus \{ \zeta \}.
    \end{align}
    Fixed any $T>0$, the following estimates hold for $x \in \mathbb{R}^n, \, \xi \in \mathbb{R}^n, \, 0 < t - \tau \leq T$:
    \begin{align}
    & \label{estimate Gamma}
        |\Gamma(z; \zeta)| \leq c(T) \, E(\xi^{-1} \circ x, \widehat{c}_5(t - \tau));\\
    & \label{estimate X_j Gamma}
        |X_j^x \Gamma(z; \zeta)| \leq c(T) \, (t - \tau)^{-\frac{1}{2}} \, E(\xi^{-1} \circ x, \widehat{c}_5(t - \tau));
    \\
    & \label{estimate X_i X_j Gamma}
        |X_i^x X_j^x \Gamma(z; \zeta)|, \, |\partial _t \Gamma(z; \zeta)| \leq c(T) \, (t - \tau)^{-1} \, E(\xi^{-1} \circ x, \widehat{c}_5(t - \tau))
    \end{align}
    with $\widehat{c}_5$ depending only on $\Lambda, \delta$, $\mathbb{G}$, and $c$ depending also on $T$.
\end{theorem}

\vspace{2mm}

\begin{proof}
Everything except (\ref{H Gamma = 0}) is straightforward from the properties of the parametrix $\Gamma_{\chi}$ and the integral term $J$ (see Theorems \ref{properties freezed gamma}-\ref{estimate freezed gamma}-\ref{Lie derivatives J}, Corollary \ref{parametrix is continuous and its 2-derivs has null integral} and Proposition \ref{continuity of J}). As for (\ref{H Gamma = 0}), for $t \leq \tau$ the thesis is immediate; for $t>\tau$, combining (\ref{X_j J}), (\ref{X_i X_j J}), (\ref{dJ/dt}) with (\ref{definition Z_1}) and the fact that
    \begin{equation*}
        \int_{t - \epsilon}^{t} d\eta \int_{\mathbb{R}^n} \left|X_j \gamma_{(y,\eta)} (y^{-1} \circ x, t - \eta) \, \mu(y,\eta; \xi,\tau) \right| \, dy \to 0 \quad \text{as} \, \epsilon \to 0^+
    \end{equation*}
    (which has been shown in the proof of Theorem \ref{Lie derivatives J}) we immediately get
    \begin{align*}
        & H \Gamma(z; \zeta) = H \Gamma_{\zeta} (z; \zeta) + \mu(z; \zeta) \, +  \lim_{\epsilon \to 0^+} \int_{\tau}^{t-\epsilon} \int_{\mathbb{R}^n} H \Gamma_{\chi} (z; \chi) \, \mu(\chi; \zeta) \, d\chi
        \\
        & \quad = Z_1(z; \zeta) + \mu(z; \zeta) \, + \lim_{\epsilon \to 0^+} \int_{\tau}^{t-\epsilon} \int_{\mathbb{R}^n} Z_1 (z; \chi) \, \mu(\chi; \zeta) \, d\chi.
    \end{align*}
    Note that, given any $z, \, \zeta \in \mathbb{R}^{n+1}$ with $t>\tau$, the map $Z_1(z;\cdot) \, \mu(\cdot; \zeta)$ belongs to $L^1(\mathbb{R}^n \times (\tau,t))$ in view of (\ref{estimate mu}), (\ref{est Z_j without lambda}), thus by Corollary \ref{eq satisfies by mu}
    \begin{align*}
        & H \Gamma(z; \zeta) = Z_1(z; \zeta) + \mu(z; \zeta) + \iint_{\mathbb{R}^n \times (\tau,t)} Z_1 (z; \chi) \, \mu(\chi; \zeta) \, d\chi = 0 .
    \end{align*}
\end{proof}

\vspace{2mm}

\section{Cauchy problem for H and properties of $\Gamma$}
\label{sec 4}

In this section we use the fundamental solution $\Gamma$ previously constructed to build a solution for the Cauchy problem

\vspace{1mm}

\begin{equation}
\label{Cauchy pb H}
    \begin{cases}
        Hu = f \hspace{1.3cm} \text{in} \, \mathbb{R}^n \times (0,T) \\
        u(\cdot,0) = g \hspace{9.6mm} \text{in} \,\mathbb{R}^n 
    \end{cases}
\end{equation}
under suitable assumptions about the regularity and growth of both $f$ and $g$ and with $T>0$ small enough.
We also exploit the result concerning the homogeneous version of (\ref{Cauchy pb H}) (see Theorem \ref{thm homogeneous Cauchy pb} below) to prove lower Gaussian estimates for $\Gamma$, thus concluding the proof of Theorem \ref{existence and gaussian estimates fund sol H}.

\vspace{2mm}

\begin{definition}
\label{def concept of solution}
    By \emph{solution} of the problem (\ref{Cauchy pb H}) we mean a function $u : \mathbb{R}^n \times (0,T) \to \mathbb{R}$ with the following properties:
    \begin{enumerate}
        \item [(i)] $u \in \mathfrak{C}^2(\mathbb{R}^n \times (0,T))$ (see Definition \ref{def of C2});
        \item [(ii)] $u$ satisfies the equation $Hu=f$ pointwise, that is \\
        $Hu(x,t) = f(x,t) \quad \forall (x,t) \in \mathbb{R}^n \times (0,T)$;
        \item[(iii)] the initial datum $g$ is assumed as a pointwise limit, namely for every $x_0 \in \mathbb{R}^n$ it holds that \\
        $u(x,t) \longrightarrow g(x_0) \quad \text{as} \:\, {d}_p((x,t),(x_0,0)) = d(x,x_0) + \sqrt{|t|} \to 0$.
    \end{enumerate}
\end{definition}

\subsection{Homogeneous Cauchy problem}
\label{subsec 4.1}
Let us start with the solution of the homogeneous version of (\ref{Cauchy pb H}):

\vspace{2mm}

\begin{theorem}[Solution of the homogeneous Cauchy problem]
\label{thm homogeneous Cauchy pb}
    Let $\nu \geq 0$ and let $g : \mathbb{R}^n \to \mathbb{R}$ be a continuous function subject to the growth condition
    \begin{equation}
    \label{growth condition g}
        |g(x)| \leq M e^{\nu  ||x||^2} \quad
        \forall \, x \in \mathbb{R}^n
    \end{equation}
    for some constant $M>0$. Then there exists $\delta > 0$ such that, for any $T>0$ satisfying $T \nu<\delta$, the function $u:\mathbb{R}^n \times (0,T] \to \mathbb{R}$ given by
    \begin{equation*}
        u(x,t) := \int_{\mathbb{R}^n} \Gamma(x,t; y,0) \, g(y) \, dy
    \end{equation*}
    belongs to $\mathfrak{C}^2 (\mathbb{R}^n \times (0,T)) \cap C(\mathbb{R}^n \times [0,T])$ and solves the Cauchy problem
    \begin{equation*}
        \begin{cases}
            Hu = 0 \hspace{1.3cm} \, \text{in} \: \mathbb{R}^n \times (0,T)\\
            u (\cdot,0) = g \hspace{1cm} \text{in} \: \mathbb{R}^n.
        \end{cases}
    \end{equation*}
    In particular, if $g$ is bounded, all the above holds for any $T>0$.
\end{theorem}

\vspace{2mm}

\begin{proof}
    By (\ref{estimate X_j Gamma}), (\ref{estimate X_i X_j Gamma}) and (\ref{growth condition g})), dominated convergence gives
    \begin{align*}
    & X_j u(x,t) = \int_{\mathbb{R}^n} (X_j^x \Gamma)(x,t; y,0) \, g(y) \, dy \quad (\forall 1 \leq j \leq m);\\
    & X_i^x X_j^x u(x,t) = \int_{\mathbb{R}^n} (X_i^x X_j^x \Gamma)(x,t; y,0) \, g(y) \, dy\quad (\forall 1 \leq i,j \leq m);\\
    & \partial _t u(x,t) = \int_{\mathbb{R}^n} (\partial _t \Gamma)(x,t; y,0) \, g(y) \, dy,
    \end{align*}
    for any $x \in \mathbb{R}^n, \, t>0$, hence $u \in \mathfrak{C}^2 (\mathbb{R}^n \times (0,T)) \cap C(\mathbb{R}^n \times (0,T])$ by (\ref{Gamma in C2}); moreover, the above and (\ref{H Gamma = 0}) yield
    \begin{equation*}
        H u(x,t) = \int_{\mathbb{R}^n} H \Gamma(x,t; y,0) \, g(y) \, dy = 0 \quad \text{in} \: \mathbb{R}^n \times (0,T).
    \end{equation*}
    We are left to discuss the initial condition. Denoting again
    \begin{align*}
        J(z; \zeta) := \iint_{\mathbb{R}^n \times (\tau,t)} \Gamma_{\chi}(z; \chi) \, \mu(\chi; \zeta) \, d\chi
    \end{align*}
    for any $z=(x,t), \, \zeta=(\xi,\tau) \in \mathbb{R}^{n+1}$ with $t>\tau$, we have
    \begin{align*}
        & |u(x,t) - g(x_0)| \leq \int_{\mathbb{R}^n} |J(x,t; y,0) \, g(y)| \, dy
        + \Bigg| \int_{\mathbb{R}^n} \gamma_{(x_0,0)} (y^{-1} \circ x, t) \, g(y) \, dy -
        \\
        &  g(x_0) \Bigg| + \int_{\mathbb{R}^n} \left| \left[\gamma_{(y,0)} (\xi^{-1} \circ x, t) - \gamma_{(x_0,0)} (y^{-1} \circ x, t)) \right] \, g(y) \right| \, dy = W_1 + W_2 + W_3.
    \end{align*}
    From (\ref{BLU (2.2)}), (\ref{estimate J}) and (\ref{growth condition g}) we have
    \begin{align*}
        & W_1 \leq c(M, T) \, \widetilde{\omega} (2 \sqrt{t}) \int_{\mathbb{R}^n} \gamma_0(y^{-1} \circ x, c_1 t) \, e^{\nu ||y||^2} \, dy 
        \nonumber \\
        & \quad \leq c(M,T) \, \widetilde{\omega}(2 \sqrt{t}) \,t^{-\frac{Q}{2}} \int_{\mathbb{R}^n} \exp\left\{ -\frac{d(x,y)^2}{c_2 t} + \nu ||y||^2 \right\} \, dy. 
    \end{align*}
    By quasi-triangle inequality $||y|| \leq \kappa \, ( ||y^{-1} \circ x|| + ||x|| )$, hence 
    \begin{align*}
        ||y||^2 \leq 2\kappa^2 \, (||y^{-1} \circ x||^2 + ||x||^2) = 2\kappa^2(d(x,y)^2+||x||^2);
    \end{align*}
    assuming $T \nu < \delta := (4 c \kappa^2)^{-1}$, one easily gets
    \begin{align*}
        -\frac{d(x,y)^2}{c_2 t} + \nu ||y||^2 \leq -\frac{d(x,y)^2}{c_2 t} \left[ 1-2c_2 t\nu \kappa^2 \right] + 2\nu ||x||^2 \leq -\frac{d(x,y)^2}{2c_2 t} + 2 \nu \kappa^2 ||x||^2
    \end{align*}
    for all $x,y \in \mathbb{R}^n$ and $0<t \leq T$, and so by (\ref{BLU (1.5)}):
    \begin{align*}
        & W_1 \leq c(M,T) \, \widetilde{\omega} (2 \sqrt{t}) \, e^{2 \nu \kappa^2  ||x||^2} \int_{\mathbb{R}^n} t^{-\frac{Q}{2}} \exp\left\{ -\frac{d(x,y)^2}{2c_2 t} \right\} \, dy \\
        & \quad \leq c(M,T) \, \widetilde{\omega}(2 \sqrt{t}) \, e^{2 \nu \kappa^2 ||x||^2}
        \longrightarrow 0 \quad \text{as} \: (x,t) \to (x_0,0).
    \end{align*}
    As for $W_2$, from \ref{hom Cauchy problem freezed operator} there holds
    \begin{equation*}
        \int_{\mathbb{R}^n} \gamma_{(x_0,0)} (y^{-1} \circ x, t) \, g(y) \, dy \longrightarrow g(x) \quad \text{as} \: t \to 0^+, \, \forall x \in \mathbb{R}^n,
    \end{equation*}
    hence it is immediate to deduce that as $(x,t) \to (x_0,0)$ we have
    \begin{equation*}
        W_2 \leq \left| \int_{\mathbb{R}^n} 
        \gamma_{(x_0,0)} (y^{-1} \circ x, t) \, g(y) \, dy - g(x) \right| + \left|g(x)-g(x_0) \right| \longrightarrow 0.
    \end{equation*}
    As for $W_3$, let $\sigma \in (0,1)$ to be chosen later; then we split
    \begin{align*}
        & W_3 = \left(\int_{d(x,y) \leq \sigma} + \int_{d(x,y) > \sigma} \right) \left| \left[\gamma_{(y,0)} (y^{-1} \circ x, t) - \gamma_{(x_0,0)} (y^{-1} \circ x, t)) \right] g(y) \right| \, dy\\
        & \quad = W_3^1 + W_3^2.
    \end{align*}
    The integration domain of $W_3^1$ is contained in the closed ball $\overline{B_r(x_0)}$ with radius $r={\kappa[d(x,x_0)+1]}$ by quasi-triangle inequality, hence $g$ is bounded (uniformly w.r.t. $\sigma$) therein. Then by (\ref{BLU (1.5)}), (\ref{BLU (2.4) with gamma}):
    \begin{align*}
        & W_3^1 \leq \check{c}_0 \int_{d(x,y)<\sigma} \omega (d(x_0,y)) \,\gamma_0 (y^{-1} \circ x, \check{c}_0 t)) \, dy \leq \check{c}_0 \omega (\kappa [d(x,x_0) + \sigma]).
    \end{align*}
   for all $\sigma \in (0,1)$. Instead, using (\ref{BLU (2.3) with gamma}) and then proceeding as for $W_1$
    \begin{align*}
        & W_3^2 \leq 2M \check{c}_0 \int_{d(x,y)>\sigma} 
        \gamma_0(y^{-1} \circ x, \check{c}_0 t) \, e^{\nu ||y||^2} \, dy
        \\
        & \quad \leq c(M,\nu,x_0) \, e^{2 \nu \kappa^2} \int_{d(x,y)>\sigma} t^{-\frac{Q}{2}} \, \exp \left\{ - \frac{d(x,y)^2}{2c_1 t} \right\} \, dy.
    \end{align*}
    Setting $y \mapsto \xi = \delta_{1/\sqrt{t}} (y^{-1} \circ x)$, which yields $d\xi = t^{-Q/2} \, dy$, we get
    \begin{align*}
        W_3^2 \leq c(M, \nu, x_0) \int_{||\xi||> \frac{\sigma}{\sqrt{t}}} \exp \left\{ - \frac{||\xi||^2}{2c_1} \right\} \, d\xi,
    \end{align*}
    which gives us the following:
    \begin{equation*}
        W_3 \leq C(M,\nu,x_0) \left[ \omega (\kappa [d(x,x_0) + \sigma]) + \int_{||\xi||> \frac{\sigma}{\sqrt{t}}} \exp \left\{ - \frac{||\xi||^2}{2c_1} \right\} \, d\xi \right].
    \end{equation*}
    Given $\epsilon > 0$, by continuity of $\omega_g$ we can choose $\delta_x>0$ such that $\omega_g (\kappa {\delta}_x) < \epsilon/4$ and then $0<\sigma<1$ such that $\omega_g (\kappa \left[{\delta}_x +{\sigma} \right]) < \epsilon/2$.
    Moreover, by Lebesgue's theorem there exists $\delta_t > 0$ such that
    \begin{equation*}
        \int_{||\xi|| > \sigma /\sqrt{\delta_t}} \, \exp \left\{ - \frac{||\xi||^2}{2c_1} \right\} \, d\xi < \frac{\epsilon}{2},
    \end{equation*}
    thus for any $(x,t) \in B(x_0, \delta_x) \times (0, \delta_t)$ it holds
    $W_3 \leq C(M,\nu,x_0) \, {\epsilon}$, which by arbitrariness of ${\epsilon}>0$ gives $W_3 \longrightarrow 0$ as $(x,t) \to (x_0,0)$.
\end{proof}

\vspace{2mm}

\subsection{Maximum principles and lower Gaussian estimates}

Let us introduce some maximum principles for the operator $H$, which will be useful both to prove the nonnegativity of the fundamental solution $\Gamma$ and to get uniqueness of solutions for the associated Cauchy problem (in suitable spaces); the latter, in turn, will reveal crucial to obtain the reproduction property (\ref{reproduction formula fund sol}) of $\Gamma$ and, consequently, also the lower Gaussian estimate in (\ref{two sided gauss est Gamma}).

\vspace{2mm}

\begin{theorem}[Weak maximum principle in bounded sets]
    \label{max principle bounded sets}
    Let $\Omega$ be an open bounded subset of $\mathbb{R}^{n+1}$ and let $t_0 \in \mathbb{R}$. Then:
    \begin{equation*}
        \begin{cases}
            u \in \mathfrak{C}^2(\Omega) \\
            Hu \leq 0 \hspace{1.45cm} \text{in} \: \Omega \cap \{t<t_0\} \\
            \limsup u \leq 0 \hspace{5mm} \text{in} \: \partial \Omega \cap \{t \leq t_0\}
        \end{cases}
        \implies \quad u \leq 0 \:\, \text{in}  \:\, \Omega \cap \{t<t_0\}.
    \end{equation*}
    (by \say{$\limsup u \leq 0$ in $\partial\Omega \cap \{t \leq t_0\}$} we mean that, for any $(\overline{x},\overline{t}) \in \partial \Omega \cap \{t \leq t_0\}$ and any sequence $\{(x_n,t_n)\}_{n \geq 1} \subset \Omega \cap \{t \leq t_0\}$ converging to $(\overline{x}, \overline{t})$, there holds $\limsup_{n \to \infty} u(x_n,t_n) \leq 0$).
\end{theorem}

\vspace{2mm}

For the proof see \cite[Theorem 4.3]{BLU_Levi_ultraparabolic} (keeping in mind that, in the latter, the opposite sign convention for $H$ is used). Denoting by $C_*(\Omega)$ the space of continuous functions in $\Omega$ which vanish at infinity, an immediate corollary is the following:

\vspace{2mm}

\begin{corollary}[Weak maximum principle in infinite strips]
\label{max principle infinite strips}
    Given $T_1<T_2$, there holds:
    \begin{equation*}
        \begin{cases}
            u \in \mathfrak{C}^2(\mathbb{R}^n \times (T_1,T_2)) \cap C_*(\mathbb{R}^n \times [T_1,T_2]) \\
            Hu=0 \hspace{1cm} \text{in} \: \mathbb{R}^n \times (T_1,T_2)\\
            u=0 \hspace{1.3cm} \text{on} \: \mathbb{R}^n \times \{T_1\}
        \end{cases} \implies
        u=0 \:\, \text{in} \:\, \mathbb{R}^n \times (T_1,T_2).
    \end{equation*}
\end{corollary}

\vspace{2mm}

\begin{proof}
    Since $u$ vanishes at infinity, for any $\epsilon>0$ there exists $M_{\epsilon}>0$ such that
    \begin{equation*}
        \sup_{||x||>M_{\epsilon}, \, t \in (T_1,T_2)} u(x,t)<\epsilon.
    \end{equation*}
    Let $M>M_{\epsilon}$, $\Omega:=\{||x||<M\} \times (T_1,T_2)$ and $t_0 \in (T_1,T_2)$. Consider the function $w_{\epsilon}(z):=u(z)-\epsilon$, $z \in \mathbb{R}^n \times (T_1,T_2)$:
    \begin{equation*}
    \begin{cases}
        w_{\epsilon} \in \mathfrak{C}^2(\mathbb{R}^n \times (T_1,T_2)); \\
        Hw_{\epsilon}=0 \quad \text{in} \: \mathbb{R}^n \times (T_1,T_2)\\
        w_{\epsilon} \leq 0 \hspace{7mm} \text{in} \: \partial \Omega \cap \{t \leq t_0\}=\{||x||=M\} \times (T_1,t_0),
    \end{cases}
    \end{equation*}
    thus $w_{\epsilon} \leq 0$ in $\{||x||<M\} \times (T_1,t_0)$  by Theorem (\ref{max principle bounded sets}). By arbitrariness of $M>M_{\epsilon}$ and $t_0 \in (T_1,T_2)$, it immediately follows that $w_{\epsilon} \leq 0$ in $\mathbb{R}^n \times (T_1,T_2)$. Then, since this holds for any $\epsilon>0$ and the sequence $\{w_{\epsilon}\}_{\epsilon>0}$ converges uniformly to $u$ in $\mathbb{R}^n \times (T_1,T_2)$ as $\epsilon \to 0^+$, we get that $u \leq 0$ in $\mathbb{R}^n \times (T_1,T_2)$. \\Note now that also $-u$ vanishes at infinity, hence all the above procedure can be analogously applied to $-u$, yielding $-u \leq 0$ in $\mathbb{R}^n \times (T_1,T_2)$, namely $u \geq 0$ in $\mathbb{R}^n \times (T_1,T_2)$. Combining the two results the thesis is achieved.
\end{proof}

\vspace{2mm}

It is immediate to deduce the following uniqueness result for solutions of (\ref{Cauchy pb H}) vanishing at infinity (here below $C_*(\Omega)$ stands for the space of continuous functions in $\Omega$ which vanish at infinity):

\vspace{2mm}

\begin{corollary}
\label{uniqueness of vanishing sol of PC}
    Let $f:\mathbb{R}^{n+1} \to \mathbb{R}$ and $g : \mathbb{R}^n \to \mathbb{R}$ be continuous functions. Then, for any $T_1<T_2$, the solution of the Cauchy problem
    \begin{equation*}
        \begin{cases}
            Hu=f \hspace{0.9cm} \text{in} \: \mathbb{R}^n \times (T_1,T_2)\\
             u(\cdot,T_1)=g \quad \text{in} \: \mathbb{R}^n,
        \end{cases}
    \end{equation*}
    if it exists, is unique in the space $\mathfrak{C}^2(\mathbb{R}^n \times (T_1,T_2)) \cap C_*(\mathbb{R}^n \times [T_1,T_2])$.
\end{corollary}

\vspace{2mm}

\begin{proof}
    Let and $v, \, w \in \mathfrak{C}^2(\mathbb{R}^n \times (T_1,T_2)) \cap C_*(\mathbb{R}^n \times [T_1,T_2])$ be both solutions of (\ref{Cauchy pb H}). Then the function $\psi := v-w$ clearly belongs to the same regularity class and satisfy $H \psi=0$ in $\mathbb{R}^n \times (T_1,T_2)$ and $\psi(\cdot,T_1) \equiv 0$; hence by Corollary \ref{max principle infinite strips} there holds $\psi \equiv 0$ in $\mathbb{R}^n \times (T_1,T_2)$, namely $v=w$ in $\mathbb{R}^n \times (T_1,T_2)$.
\end{proof}

\vspace{2mm}

Coming back to the fundamental solution $\Gamma$, we are now ready to show that:

\vspace{2mm}

\begin{theorem}
\label{nonnegativity and reproduction}
    The fundamental solution $\Gamma$ is nonnegative in $(\mathbb{R}^{n+1} \times \mathbb{R}^{n+1}) \setminus D$ and the reproduction property (\ref{reproduction formula fund sol}) holds true.
\end{theorem}

\vspace{2mm}

\begin{proof}
   Assume by contradiction that $\Gamma(x_0,t_0;\xi_0,\tau_0)<0$ for some $(x_0,t_0;\xi_0,\tau_0) \notin D$; then $\Gamma(x_0,t_0; \cdot, \tau_0)<0$ in $B_{\epsilon}(\xi_0)$ for some $\epsilon>0$, by continuity of $\Gamma$. Let $T>0$ be small enough and $g \neq 0$ be a nonnegative smooth function in $\mathbb{R}^n$ with support in $B_{\epsilon}(\xi_0)$; then, by Theorem \ref{thm homogeneous Cauchy pb}, the function
    \begin{equation*}
        u(x,t)=\int_{\mathbb{R}^n} \Gamma\left(x,t; \xi,\tau_0-\frac{T}{2}\right) \, g(\xi) \, d\xi, \quad (x,t) \in \mathbb{R}^{n+1} 
    \end{equation*}
    is well defined and solves 
    \begin{equation*}
        \begin{cases}
            Hu=0 \hspace{8mm} \text{in} \: \mathbb{R}^n \times (\tau_0-\frac{T}{2},\tau_0+\frac{T}{2})\\
            u(\cdot,\tau_0-\frac{T}{2})=g \quad \text{in} \: \mathbb{R}^n.
        \end{cases}
    \end{equation*}
    Moreover, by dominated convergence and Theorem \ref{properties freezed gamma}
    \begin{align*}
        \lim_{||x|| \to +\infty} u(x,t)=\int_{\mathbb{R}^n} \lim_{||x|| \to +\infty} \Gamma\left(x,t; \xi,\tau_0-\frac{T}{2} \right) \, g(\xi) \, d\xi =0.
    \end{align*}
    Hence $Hu=0$ in $\mathbb{R}^n \times (\tau_0-T/2,\tau_0+T/2)$ and $\liminf u \geq 0$ on $\mathbb{R}^n \times \{\tau_0-T/2\}$ and at infinity, which (by Corollary \ref{max principle infinite strips}) yields $u \geq 0$ in $\mathbb{R}^n \times (\tau_0-T/2,\tau_0+T/2)$.\\
    However, since $supp(g) \subset B_{\epsilon}(\xi_0)$ but $g$ does not vanish everywhere, we have
    \begin{equation*}
        u(x_0,t_0)=\int_{B_{\epsilon}(\xi_0)} \Gamma\left(x_0,t_0; \xi,\tau_0 - \frac{T}{2} \right) g(\xi) \, d\xi <0,
    \end{equation*}
    which is a contradiction. Thus $\Gamma$ is nonnegative outside $D$.
    As for the reproduction property (\ref{reproduction formula fund sol}), let us fix $s \in \mathbb{R}, \, \xi \in \mathbb{R}^n, \, \tau<s$; consider the Cauchy problem
    \begin{equation*}
        \begin{cases}
            Hu=0 \hspace{1cm} \text{in} \: \mathbb{R}^n \times (s,s+T),\\
            u(\cdot,s)=g \hspace{6mm} \text{in} \: \mathbb{R}^n,
        \end{cases}
    \end{equation*}
    with $g=\Gamma(\cdot,s; \xi,\tau)$. Note that, since $s>\tau$, then $g$ is well defined and continuous on $\mathbb{R}^n$; moreover, in view of (\ref{BLU (2.2)}) and the fact that $s$ and $\tau$ are fixed and distinct, $g$ is also bounded. By Theorem \ref{thm homogeneous Cauchy pb} it follows that the function
    \begin{equation*}
        \psi(x,t; s,\xi,\tau):=\int_{\mathbb{R}^n} \Gamma(x,t; y,s) \, g(y) \, dy =\int_{\mathbb{R}^n} \Gamma(x,t; y,s) \, \Gamma(y,s; \xi,\tau) \, dy
    \end{equation*}
    is well defined for all $x, \, \xi \in \mathbb{R}^n, \, t>\tau$, belongs to $\mathfrak{C}^2(\mathbb{R}^n \times (s,s+T))$ and solves the Cauchy problem above. Furthermore, by (\ref{estimate Gamma}) and dominated convergence
    \begin{align*}
        \lim_{||x|| \to +\infty} \psi(x,t; s, \xi,\tau) =  \int_{\mathbb{R}^n} \lim_{||x|| \to +\infty} \Gamma(x,t;y,s) \, \Gamma(y,s; \xi,\tau) \, dy =0.
    \end{align*}
    However also $\Gamma(x,t; \xi,\tau) \in \mathfrak{C}^2(\mathbb{R}^n \times (s,s+T))$, solves the Cauchy problem and vanishes at infinity (being the fundamental solution of $H$); so by Theorem \ref{uniqueness of vanishing sol of PC}
    \begin{equation*}
        \Gamma(x,t;\xi,\tau)=\psi(x,t; s,\xi,\tau)=\int_{\mathbb{R}^n} \Gamma(x,t; y,s) \, \Gamma(y,s; \xi,\tau) \, dy
    \end{equation*}
    for all $x, \, \xi \in \mathbb{R}^n$, $\tau<s<t<s+T$. By arbitrariness of $T>0$, we get (\ref{reproduction formula fund sol}).
\end{proof}

\vspace{2mm}

Let us now focus on proving a Gaussian estimate from below for the fundamental solution $\Gamma$: first we prove it in local form, namely for points $(z; \zeta)$ \say{sufficiently close} to the diagonal $D$; then we extend the result to the whole space by (\ref{reproduction formula fund sol}), following a geometric idea used in \cite[Theorem 2.2]{Harnack_inequality}.

\vspace{2mm}

\begin{theorem}[Lower Gaussian estimate for the fundamental solution]
\label{lower gaussian estimate gamma}
    For any $T>0$ there exists a constant $\overline{\overline{c}}=\overline{\overline{c}}(T)>0$ such that, for every $(x,t), (\xi,\tau) \in \mathbb{R}^{n+1}$ with $0<t-\tau \leq T$, the following uniform Gaussian estimate holds true:
    \begin{equation*}
        \Gamma(x,t; \xi,\tau) \geq \overline{\overline{c}}^{-1} (t-\tau)^{-Q/2} \exp\left\{ -\overline{\overline{c}} \, \frac{d(x,\xi)^2}{t-\tau} \right\}.
    \end{equation*}
\end{theorem}

\vspace{2mm}

\begin{proof}
    Recall that $\Gamma(x,t; \xi,\tau) = \Gamma_{(\xi,\tau)}(x,t; \xi,\tau) + J(x,t; \xi,\tau)$, with the former term at right hand side bounded below by (\ref{BLU (2.2)}) and the latter bounded above by (\ref{estimate J}); by equivalence between $d$ and the CC distance $d_X$, namely 
    \begin{equation*}
        \beta^{-1} d_X(x,\xi)^2 \leq d(x,\xi)^2 \leq \beta d_X(x,\xi)^2 \quad \forall x,\xi \in \mathbb{R}^n
    \end{equation*}
    for a constant $\beta \geq 1$, the mentioned estimates can be written as
    \begin{align*}
        & \Gamma_{(\xi,\tau)}(x,t;\xi,\tau) \geq \overline{c}_0^{-1} \, (t-\tau)^{-Q/2} \, \exp\left\{ -\overline{c}_0 \beta \frac{d_X(x,\xi)^2}{t-\tau}\right\}; \\
        & |J(x,t; \xi,\tau)| \leq c(t-\tau)^{-Q/2} \,  \widetilde{\omega}(2 \sqrt{t-\tau}) \, \exp\left\{ -\frac{d_X(x,\xi)^2}{\widehat{c}_3 \beta (t-\tau)} \right\}
    \end{align*}
    (with the advantage that $d_X$ is a true distance). These yield
    \begin{align*}
        & \Gamma(x,t; \xi,\tau) \geq \overline{c}_0^{-1} (t-\tau)^{-Q/2} \, \exp\left\{  -\overline{c}_0 \frac{d_X(x,\xi)^2}{t-\tau} \right\} \left[1 - \overline{c}_0 c \, \widetilde{\omega}(2 \sqrt{t-\tau}) \times \right.\\
        & \quad \times \left. \exp\left\{ \left(\overline{c}_0-\frac{1}{\widehat{c}_3}\right) \frac{d_X(x,\xi)^2}{t-\tau} \right\} \right] =: \overline{c}_0^{-1} (t-\tau)^{-Q/2} \, \exp\left\{  -\overline{c}_0\frac{d_X(x,\xi)^2}{t-\tau} \right\} \times \\
        & \quad \times \left[1 - c_2 \widetilde{\omega}(2 \sqrt{t-\tau}) \exp\left\{ c_3 \frac{d_X(x,\xi)^2}{t-\tau}\right\} \right],
    \end{align*}
    We choose $0<\delta<\frac{1}{4}$ such that $c_3 \delta<1$ and $\widetilde{\omega}(2 \sqrt{\delta}) \leq \frac{1}{4c_2}$,
   which is clearly possible in view of the properties of a modulus of continuity. Then for any $x,t,\xi,\tau$ with $0<t-\tau \leq \delta$ and $d_X(x,\xi)^2 \leq \delta (t-\tau)$ there holds
    \begin{align*}
        1-c_2 \widetilde{\omega}(2 \sqrt{t-\tau}) \exp\left\{ c_3 \frac{d_X(x,\xi)^2}{t-\tau} \right\} \geq 1-c_2 \, \widetilde{\omega}(2 \sqrt{\delta}) \, e^{c_3 \delta}>\frac{1}{4},
    \end{align*}
    thus
    \begin{align*}
        & \Gamma(x,t; \xi,\tau) \geq (4\overline{c}_0)^{-1} (t-\tau)^{-Q/2} \exp\left\{ -4\overline{c}_0 \frac{d_X(x,\xi)^2}{t-\tau} \right\}.
    \end{align*}
    Now, fixed any $T>0$, we show an analogous bound for all points $(x,t), (\xi,\tau) \in \mathbb{R}^{n+1}$ with $0<t-\tau \leq T$; it is not restrictive to assume $T \geq 1$. Let 
    \begin{equation*}
        k := \left\lfloor \max\left\{ \frac{T}{\delta}, \, \frac{4d_X(x,\xi)^2}{\delta(t-\tau)}\right\} \right\rfloor + 1, \quad \sigma:=\frac{1}{4} \sqrt{\delta \frac{t-\tau}{k+1}}.
    \end{equation*}
    Given $\epsilon>0$, take an absolutely continuous curve $\gamma : [0,1] \to \mathbb{R}^n$ joining $x$ to $\xi$ with length smaller than $d_X(x,\xi)+\epsilon$, a chain of points $x \equiv x_0, \, x_1, \, ..., \, x_{k+1} \equiv \xi$ on $\gamma$ and a chain of times $t \equiv t_0>t_1>...>t_{k+1} \equiv \tau$ with
    \begin{align*}
        & d_X(x_j,x_{j+1})=\frac{d(x,\xi)+\epsilon}{k+1}, \quad t_j-t_{j+1} = \frac{t-\tau}{k+1} \quad (\forall 0 \leq j \leq k).
    \end{align*}
    Then, for any $0 \leq j \leq k$ and any $y \in B_{d_X}(x_j,\sigma), \, y' \in B_{d_X}(x_{j+1},\sigma)$, there holds
    \begin{align*}
        & d_X(y,y') \leq d_X(y,x_j)+d_X(x_j,x_{j+1})+d_X(x_{j+1},y') \leq \frac{d_X(x,\xi)+\epsilon}{k+1}+2\sigma\\
        & \quad < \frac{\frac{1}{2} \sqrt{k \delta (t-\tau)}}{k+1} + \frac{1}{2}\sqrt{\delta\frac{t-\tau}{k+1}}<\sqrt{\delta \frac{t-\tau}{k+1}}= \sqrt{\delta(t_j-t_{j+1})}
    \end{align*}
    (provided $\epsilon$ is small enough); since we also have $t_j-t_{j+1} \leq T/(k+1) < \delta$, then the lower Gaussian estimate holds for points $(y,t_j), (y',t_{j+1})$, i.e.
    \begin{align*}
        \Gamma(y,t_j;y',t_{j+1}) \geq c^{-1} (t_j-t_{j+1})^{-Q/2} \exp\left\{ -c \frac{d_X(y,y')^2}{t_j-t_{j+1}} \right\}
    \end{align*}
    for any constant $c \geq 4\overline{c}_0$. Setting $y_0=x$ and $y_{k+1}=\xi$ for homogeneity of notation), we iterate (\ref{reproduction formula fund sol}) and exploit the nonnegativity of $\Gamma$: 
    \begin{align*}
        & \Gamma(x,t; \xi,\tau)=\int_{(\mathbb{R}^n)^k} \Gamma(y_0,t_0; y_1,t_1) \, \Gamma(y_1,t_1; y_2,t_2) ... \Gamma(y_k,t_k; y_{k+1},t_{k+1}) \, dy_1 ... dy_k\\
        & \quad \geq \int_{\Omega_{\sigma}:=B_{d_X}(x_1,\sigma) \times ... \times B_{d_X}(x_k,\sigma)} \Gamma(x_0,t_0; y_1,t_1) ... \Gamma(y_k,t_k; x_{k+1},t_{k+1}) \, dy_1 ... dy_k\\
        & \quad \geq c^{-k-1}\left( \frac{t-\tau}{k+1} \right)^{-\frac{Q}{2}(k+1)} \int_{\Omega_{\sigma}} \exp\left\{ -c\sum_{j=0}^k \frac{d_X(y_j,y_{j+1})^2}{t_j-t_{j+1}} \right\} \, dy_1...dy_k\\
        & \quad \geq c^{-k-1}\left( \frac{t-\tau}{k+1} \right)^{-\frac{Q}{2}(k+1)} \exp\left\{ -c \delta (k+1) \right\} \, \left[|B_{d_X}(0,1) \right| \, \sigma^Q]^k \\
        & \quad = c^{-k-1} (t-\tau)^{-Q/2} (k+1)^{Q/2} |B_{d_X}(0,1)|^k 4^{-Qk} \,\exp\left\{ -c\delta (k+1) \right\} \\
        & \quad \geq c^{-1} (t-\tau)^{-
        Q/2} \exp\left\{-ck\right\} 
        \exp\left\{-c\delta(k+1)+ck-
        k\log\left(\frac{c \cdot 4^Q}
        {|B_{d_X}(0,1)|}\right)\right\}.
    \end{align*}
    We fix $c$ to a value $\overline{c}$ large enough for the argument of the last exponential to be positive whenever $k \geq 1$, so that $\Gamma(x,t; \xi,\tau) \geq \overline{c}^{-1}(t-\tau)^{-Q/2} \exp\{-\overline{c}k\}$; finally, we split in two cases depending on which one among $\frac{T}{\delta}$ and $\frac{4d_X(x,\xi)^2}{\delta(t-\tau)}$ is bigger: 
    \begin{itemize}
        \item [(i)] if $\frac{4d_X(x,\xi)^2}{t-\tau} \leq T$, then $k=\left\lfloor \frac{T}{\delta} \right\rfloor+1< 2 \frac{T}{\delta}$, so
        \begin{align*}
            & \Gamma(x,t; \xi,\tau) \geq \overline{c}^{-1} (t-\tau)^{-Q/2} \exp\left\{ -2\overline{c}T \delta^{-1} \right\} = [\overline{c} e^{-2\overline{c}T/\delta}]^{-1} (t-\tau)^{-Q/2} \\
            & \quad \geq [\overline{c} e^{-2\overline{c}T/\delta}]^{-1} (t-\tau)^{-Q/2} \exp\left\{ -[\overline{c} e^{-2\overline{c}T/\delta}] \frac{d(x,\xi)^2}{t-\tau} \right\};
        \end{align*}
        \item [(ii)] otherwise $\frac{4d_X(x,\xi)^2}{t-\tau} >
        T \geq 1$ and $k=\left\lfloor \frac{4d_X(x,\xi)^2}{\delta(t-\tau)} \right\rfloor +1 < \frac{8d_X(x,\xi)^2}{\delta(t-\tau)}$, thus (exploiting again equivalence between $d$ and $d_X$):
        \begin{align*}
            & \Gamma(x,t; \xi,\tau) \geq \overline{c}^{-1} (t-\tau)^{-Q/2} \exp\left\{ -\overline{c}\frac{8\beta d(x,\xi)^2}{\delta(t-\tau)} \right\} \\
            & \quad \geq [8\overline{c}\beta \delta^{-1}]^{-1} (t-\tau)^{-Q/2} \exp\left\{ -[8\overline{c}\beta  \delta^{-1}] \frac{d(x,\xi)^2}{t-\tau} \right\},
        \end{align*}
    \end{itemize}
    hence the thesis holds with $\overline{\overline{c}}=\overline{\overline{c}}(T)=\max\{\overline{c} e^{-2\overline{c}T/\delta}, \, 8\overline{c} \beta \delta^{-1}\}$.
\end{proof}

\vspace{2mm}

\subsection{Nonhomogeneous Cauchy problem}
\label{subsec 4.2}
\begin{lemma}
\label{regularity of V_f}
Let us assume that $\omega$ satisfies also the condition (H3). Let $\nu \geq 0$, let $T>0$ be small enough. Let $f \in C(\mathbb{R}^n \times [0,T])$ be locally Dini continuous in $x$ uniformly w.r.t. $t$, namely for any $K \Subset \mathbb{R}^N$ the modulus of continuity
\begin{equation*}
    \omega_{f,K}(r):=\sup_{x,y \in K: \, d(x,y) < r} |f(x)-f(y)|
\end{equation*}
satisfies
\begin{align*}
    \sup_{t \in [0,T]} \int_{0}^r \frac{\omega_{f,K}(x)}{x} \, dx<+\infty.
\end{align*}
Assume moreover that $f$ is subject to the growth condition
    \begin{equation}
    \label{growth condition f}
        |f(x,t)| \leq M e^{\nu ||x||^2} \quad
        \forall \, x \in \mathbb{R}^n, \, t \in [0,T]
    \end{equation}
    for some constant $M>0$. Then the function $V_f:\mathbb{R}^n \times [0,T] \to \mathbb{R}$ given by
    \begin{equation*}
        V_f(x,t) = \iint_{\mathbb{R}^n \times (0,t)} \Gamma_{\zeta}(x,t;\zeta) f(\zeta) \, d\zeta \quad \text{if} \: t>0, \quad V_f(x,0)=0.
    \end{equation*}
    is well defined, belongs to $\mathfrak{C}^2(\mathbb{R}^n \times (0,T)) \cap C(\mathbb{R}^n \times [0,T])$ and
    \begin{equation}
    \label{H V_f}
        HV_f(x,t) = f(x,t) + \iint_{\mathbb{R}^n \times (0,t)} Z_1(x,t;\zeta) f(\zeta) \, d\zeta \quad \forall (x,t) \in \mathbb{R}^n \times (0,T).
    \end{equation}
\end{lemma}

\vspace{2mm}

\begin{proof}
Note that, except for the continuity of $V_f$ at $t=0$, the statement is the same as that of Theorem \ref{Lie derivatives J}, with $\tau=0$ and the integral map $\mu(\cdot;\zeta)$ replaced by $f$; thus, up to substituting the inequality (\ref{estimate mu}) with (\ref{growth condition f}) and the Dini-type estimate (\ref{estimate delta mu}) with the local Dini continuity of $f(\cdot,t)$, the same arguments can be repeated, yielding $V_f \in \mathfrak{C}^2(\mathbb{R}^n \times (0,T)) \cap C(\mathbb{R}^n \times (0,T])$ with
\begin{align*}
    HV_f(x,t)=f(x,t)+\iint_{\mathbb{R}^n \times (0,t)} [H\Gamma(\cdot;\zeta)](x,t) \, f(\zeta) \, d\zeta = f(x,t).
\end{align*}
    We are left to prove that $V_f(\cdot,t) \to 0$ as $t \to 0^+$; by quasi-triangle inequality
    \begin{align*}
        |f(\xi,\tau)| \leq Me^{\nu ||\xi||^2} \leq M e^{2 \nu \kappa^2 (||x||^2 + d(x,\xi)^2)} 
    \end{align*}
    for all $x,\xi \in \mathbb{R}^n$ and $0<t-\tau \leq T$; so, in view also of (\ref{BLU (1.5)}), (\ref{BLU (2.3) with gamma}) we have
    \begin{align*}
        & 0 \leq |V_f(x,t)| \leq \check{c}_0 M e^{2 \nu \kappa^2 ||x||^2} \int_0^t ds \int_{\mathbb{R}^n}  \frac{\exp\left\{-\frac{d(x,\xi)^2}{\check{c}_0(t-\tau)}\left[ 1-2\nu \kappa^2 \check{c}_0 T \right] \right\}}{(t-\tau)^{Q/2}} d\xi.
    \end{align*}
    Thus, assuming $4 \nu \kappa^2 \check{c}_0 T<1$ (i.e., $T$ small enough), we get
    \begin{align*}
        & 0 \leq |V_f(x,t)| \leq \check{c}_0 M e^{2 \nu \kappa^2 ||x||^2} \int_0^t  \int_{\mathbb{R}^n} (t-\tau)^{-\frac{Q}{2}} \exp\left\{ -\frac{d(x,\xi)^2}{2\check{c}_0(t-\tau)} \right\} d\xi
        \\
        & \quad \leq \check{c}_0 M e^{2 \nu \kappa^2 ||x||^2} \int_0^t ds \int_{\mathbb{R}^n} \gamma_0(\xi^{-1} \circ x, 2\check{c}_0 (t-\tau)) \, d\xi = c(M) e^{2 \nu \kappa^2 ||x||^2} t \to 0
    \end{align*}
    as $t \to 0^+$ (for every $x \in \mathbb{R}^n$), which concludes the proof.
\end{proof}

\vspace{2mm}

\begin{lemma}
\label{regularity f tilde}
    Let $f : \mathbb{R}^n \times [0,T] \to \mathbb{R}$ be a continuous function subject to the growth condition (\ref{growth condition f}) for some constants $\nu>0, \, M>0$. Then the function
    \begin{equation*}
        \widetilde{f}(z) := \iint_{\mathbb{R}^n \times (0,t)} \mu(z; \zeta) f(\zeta) \, d\zeta, \quad z \in \mathbb{R}^n \times [0,T]
    \end{equation*}
    is well-defined, continuous w.r.t. $(x,t)$ and locally Dini continuous in $x$, uniformly w.r.t. $t \in [0,T]$. Furthermore, $\widetilde{f}$ is subject to the growth condition
    \begin{equation*}
        |\widetilde{f}(x,t)| \leq c \, \widetilde{\omega}(2\sqrt{T}) e^{2 \nu \kappa^2 ||x||^2} \quad \forall (x,t) \in \mathbb{R}^n \times [0,T]
    \end{equation*}
    for a constant $c=c(M,T)>0$.
\end{lemma}

\vspace{2mm}

\begin{proof}
    Given $T>0$, $K \Subset \mathbb{R}^n $ (compact) and fixed $x,x' \in K$ and $t,t+h \in (0,T)$, we study the oscillation of $\widetilde{f}$ in space and time separately. As for the former:
    \begin{align*}
        |\widetilde{f}(x,t)-\widetilde{f}(x',t)| \leq \iint_{\mathbb{R}^n \times (0,t)} |\mu(x,t;\zeta)-\mu(x',t;\zeta)| \cdot |f(\zeta)| \, d\zeta.
    \end{align*}
    The use of quasi-triangle inequality as in Theorem \ref{regularity of V_f} for both $x,x'$ gives
    \begin{align*}
        |f(\xi,\tau)| \leq M e^{\nu ||\xi||^2} \leq M e^{\nu \kappa^2 \cdot \min\{||x||^2+d(x,\xi)^2, \, ||x'||^2 + d(x',\xi)^2\}};
    \end{align*}
    moreover, starting from either (\ref{estimate delta mu}) if $t-\tau \geq d(x,x')^2$ or (\ref{estimate mu}) if $t-\tau < d(x,x')^2$, by monotonicity of $\omega$ and (\ref{BLU (2.2)}) one gets
    \begin{align*}
    & |\mu(x',t;\xi,\tau) - \mu(x,t;\xi,\tau)| \leq c(T) \, (t-\tau)^{-1-\frac{Q}{2}} \, \left[\omega^{1/2}(2d(x,x')) \, \omega^{1/2}(2\sqrt{t-\tau}) \right.
    \\
    & \quad \left.+ \widetilde{\omega}(2\sqrt{t-\tau}) \, \widetilde{\omega}(2d(x,x')) \right] \cdot \left[\exp\left\{ -\frac{d(x,\xi)^2}{\widetilde{c}(t-\tau)} \right\} + \exp\left\{ -\frac{d(x',\xi)^2}{\widetilde{c}(t-\tau)} \right\} \right]
\end{align*}
provided $0<t-\tau \leq T$, for a constant $\widetilde{c}>0$. In compact notation:
\begin{align*}
    & |\mu(x,t;\zeta)-\mu(x',t;\zeta)| \leq c(T) (t-\tau)^{-1-\frac{Q}{2}} \omega_1(d(x,x')) \, \omega_1(\sqrt{t-\tau})[\Phi(x)+\Phi(x')], \\
    & \quad \text{where} \: \omega_1(r):=\omega^{1/2}(2r)+\widetilde{\omega}(2r), \quad \Phi:=\exp\left\{ -\frac{d(\cdot,\xi)^2}{\widetilde{c}(t-\tau)} \right\}
\end{align*}
Note that by (H2)-(H3) $\omega_1$ satisfies the Dini condition, and provided $T$ is small enough (i.e. $T<(4\widetilde{c} \nu \kappa^2)^{-1}$, as for the previous proof) we have
\begin{align*}
    \Phi(\cdot) \, |f(\xi,\tau)| \leq M \exp\left\{ -\frac{d(\cdot,\xi)^2}{\widetilde{c}(t-\tau)} (1-\widetilde{c} \nu \kappa^2 T) \right\} \leq M \exp\left\{ -\frac{d(\cdot,\xi)^2}{2\widetilde{c}(t-\tau)} \right\}
\end{align*}
in both $x,x'$; all the above combined with (\ref{BLU (1.5)}), (\ref{BLU (2.2)}) yield
    \begin{align}
        & \nonumber |\widetilde{f}(x',t) - \widetilde{f}(x,t)| \leq c_K  \omega_1(d(x,x')) \int_{0}^{t} \frac{\omega_1(\sqrt{t-\tau})}{t-\tau} \int_{\mathbb{R}^n} (t-\tau)^{-\frac{Q}{2}} \times \\
        & \nonumber \quad \times \sum_{y \in \{x,x'\}} \exp\left\{ -\frac{d(y,\xi)^2}{2\widetilde{c}(t-\tau)}\right\} d\xi \leq c_K \omega_1(d(x,x')) \times \int_{0}^{t} \frac{\omega_1(\sqrt{t-\tau})}{t-\tau} \, d\tau \\
        & \quad \label{space variation f tilde} \int_{\mathbb{R}^n} \sum_{y \in \{x,x'\}} \gamma_0(\xi^{-1} \circ y, \overline{c}_0 \widetilde{c}(t-\tau)) \, d\xi \leq c_K \omega_1(d(x,x')) \, 
    \end{align}
    for a constant $c_K=c(T,M,\nu,K)$,
    that is the $t$-uniform local Dini continuity of $\widetilde{f}(\cdot,t)$. We deduce the global continuity studying also the time oscillation:
    \begin{align*}
        & \widetilde{f}(x',t+h)-\widetilde{f}(x',t) = \iint_{\mathbb{R}^n \times (t,t+h)} \mu(x',t+h;\zeta) \, f(\zeta) \, d\zeta + 
        \\
        & \quad + \iint_{\mathbb{R}^n \times (0,t)} (\mu(x',t+h;\zeta) - \mu(x',t; \zeta)) \, f(\zeta) \, d\zeta = B_1(h) + B_2(h).
    \end{align*}
    Since $\mu(x',\cdot; \zeta)$ is continuous, as $h \to 0$ $B_2(h)$ vanishes by dominated convergence; moreover, by (\ref{BLU (1.5)}), (\ref{estimate mu}), (\ref{growth condition f}) and treating again the exponential $e^{\nu ||\xi||^2}$ via quasi-triangle inequality, one gets
 \begin{align*}
     & |B_1(h)| \leq c(T,M) e^{2\nu \kappa^2 ||x'||^2} \int_{t}^{t+h} \frac{\omega(2\sqrt{t+h-\tau})}{t+h-\tau} \, d\tau 
     \\
     & \quad \int_{\mathbb{R}^n} \gamma_0(\xi^{-1} \circ x', \overline{c}_0 \widetilde{c}(t+h-\tau)) \, d\xi = c(T,M) e^{2 \nu \kappa^2 ||x'||^2} \widetilde{\omega}(2\sqrt{h}) \longrightarrow 0 \quad \text{as} \: h \to 0,
 \end{align*}
 which combined with (\ref{space variation f tilde}) gives the continuity of $\widetilde{f}$ over $\mathbb{R}^n \times [0,T]$. Finally, repeating the last integral computation but with $\tau$ ranging in $(0,t)$, we get
    \begin{align*}
        & |\widetilde{f}(x,t)| \leq c(T,M) \, e^{2\nu \kappa^2 ||x||^2} \int_0^t \frac{\omega(2\sqrt{t-\tau})}{t-\tau} \, d\tau \leq c(T,M) \, \widetilde{\omega}(2 \sqrt{T}) \, e^{2\nu \kappa^2 ||x||^2}
    \end{align*}
    for all $x \in \mathbb{R}^n$ and $t>0$.
\end{proof}

\vspace{2mm}

\begin{theorem}
\label{solution non hom Cauchy pb, case g=0}
    Under the assumptions of Lemma \ref{regularity of V_f}, the function
    \begin{equation*}
        u(z) := \iint_{\mathbb{R}^n \times (0,t)} \Gamma(z; \zeta) \, f(\zeta) \, d\zeta \quad \forall z \in \mathbb{R}^n \times [0,T]
    \end{equation*}
    belongs to $\mathfrak{C}^2(\mathbb{R}^n \times (0,T)) \cap C(\mathbb{R}^n \times [0,T])$ and solves the Cauchy problem
    \vspace{1mm}
    \begin{equation*}
        \begin{cases}
            Hu=f \hspace{1.13cm} \text{in} \: \mathbb{R}^n \times (0,T) \\
            u(\cdot,0) = 0 \hspace{8mm} \text{in} \: \mathbb{R}^n.
        \end{cases}
    \end{equation*}
\end{theorem}

\vspace{2mm}

\begin{proof}
    First of all, exploiting the decomposition (\ref{def Gamma Levi method}) we get
    \begin{align*}
        & u(z) = \iint_{\mathbb{R}^n \times (0,t)} \Gamma(z; \zeta) \, f(\zeta) \, d\zeta = \iint_{\mathbb{R}^n \times (0,t)} (\Gamma_{\zeta}(z; \zeta) + J(z;\zeta)) \, f(\zeta) \, d\zeta
        \\
        & = V_f(z) + \iint_{\mathbb{R}^n \times (0,t)} \left[ \iint_{\mathbb{R}^n \times (\tau,t)} \Gamma_{\chi}(z; \chi) \, \mu(\chi; \zeta) \, d\chi \right] \, f(\zeta) \, d\zeta
        \\
        & = V_f(z) + \iint_{\mathbb{R}^n \times (0,t)} \Gamma_{\chi}(z; \chi) \, \left[ \iint_{\mathbb{R}^n \times (0,\eta)} \mu(\chi; \zeta) \, f(\zeta) \, d\zeta \right] \, d\chi
        \\
        & = V_f(z) + \iint_{\mathbb{R}^n \times (0,t)} \Gamma_{\chi}(z; \chi) \, \widetilde{f}(\chi) \, d\chi = V_f(z) + V_{\widetilde{f}}(z).
    \end{align*}
    Applying Lemma \ref{regularity of V_f} on $\widetilde{f}$ (in view of Lemma \ref{regularity f tilde}) and exploiting Corollary \ref{eq satisfies by mu}
    \begin{align*}
        & H V_{\widetilde{f}}(z) =\widetilde{f}(z) + \iint_{\mathbb{R}^n \times (0,t)} Z_1(z;\zeta) \, \widetilde{f}(\zeta) \, d\zeta = \iint_{\mathbb{R}^n \times (0,t)} f(\chi) \mu(z;\chi) \, d\chi
        \\
        & \quad + \iint_{\mathbb{R}^n \times (0,t)} Z_1(z;\zeta) \left( \iint_{\mathbb{R}^n \times (0,\tau)} f(\chi) \mu(\zeta; \chi) \, d\chi \right) d\zeta\\
        & \quad = \iint_{\mathbb{R}^n \times (0,t)} f(\chi)  \left[\mu(z;\chi) + \iint_{\mathbb{R}^n \times (\eta,t)} Z_1(z;\zeta) \mu(\zeta; \chi) \, d\zeta \right] d\chi 
        \\
        & \quad = -\iint_{\mathbb{R}^n \times (0,t)} f(\chi) Z_1(z;\chi) \,  d\chi.
    \end{align*}
    Combining this last result with (\ref{H V_f}) we get $Hu = f \:\, \text{in} \, \mathbb{R}^n \times (0,T)$; the attainment of the initial condition follows directly from definition of $V_f$ and $V_{\widetilde{f}}$.
\end{proof}

\vspace{2mm}

\begin{theorem}
\label{solution general Cauchy pb}
    Under the assumptions of Theorem \ref{thm homogeneous Cauchy pb} and Lemma \ref{regularity of V_f}, the function 
    \begin{equation*}
        u(z) := \iint_{\mathbb{R}^n \times (0,t)} \Gamma(z; \zeta) f(\zeta) \, d\zeta + \int_{\mathbb{R}^n} \Gamma(z; \xi,0) g(\xi) \, d\xi, \quad z \in \mathbb{R}^n \times [0,T]
    \end{equation*}
    belongs to $\mathfrak{C}^2(\mathbb{R}^n \times (0,T)) \cap C(\mathbb{R}^n \times [0,T])$ and solves the Cauchy problem (\ref{Cauchy pb H}).
\end{theorem}

\vspace{2mm}

\begin{proof}
    The thesis is a direct consequence of Theorem \ref{thm homogeneous Cauchy pb} and Theorem \ref{solution non hom Cauchy pb, case g=0}, thanks to the linearity of the Cauchy problem (\ref{Cauchy pb H}).
\end{proof}

\vspace{1cm}

\begin{address}
    M. Faini: Dipartimento di Matematica, Politecnico di Milano, via Bonardi 9, 20133 Milano, Italy.\\
    e-mail: matteo.faini@polimi.it
\end{address}

\vspace{2mm}

\addtocontents{toc}{\vspace{2em}} 

\end{document}